\documentclass{article}

\usepackage{amsmath}
\usepackage{amsthm}
\usepackage{amsfonts}
\usepackage{latexsym}

\newcommand{\eref}[1]{(\ref{#1})}

\newcommand{\R}{\mathbb{R}}
\newcommand{\hh}{{{\dot{\mathrm{H}}^{s_p}}\times{\dot{\mathrm{H}}^{s_p-1}}}}
\newcommand{\tpu}{{T_+((u_0,u_1))}}
\newcommand{\tmu}{{T_-((u_0,u_1))}}
\newcommand{\rr}{_{rad}}
\newcommand{\supp}{\mathrm{supp}\,}

\newcommand{\n}[1]{\left\Vert#1\right\Vert} 
\newcommand{\nl}[2]{\left\Vert#1\right\Vert_{L_x^{#2}}} 
\newcommand{\nh}[1]{\left\Vert#1\right\Vert_{\dot H^{s_p}}}
\newcommand{\nH}[1]{\left\Vert#1\right\Vert_{\dot H^{s_p-1}}}

\newcommand{\nhh}[2]{{\left\Vert\left(#1,#2\right)\right\Vert_{\hh}}} 
\newcommand{\nhhx}[1]{{\left\Vert#1\right\Vert_{\hh}}} 
\newcommand{\nll}[3]{\left\Vert#1\right\Vert_{L_t^{#2}L_x^{#3}}} 
\newcommand{\nlli}[3]{\left\Vert#1\right\Vert_{L_I^{#2}L_x^{#3}}} 
\newcommand{\nllij}[3]{\left\Vert#1\right\Vert_{L_{I_j}^{#2}L_x^{#3}}} 
\newcommand{\ns}[2]{\left\Vert#1\right\Vert_{S_p\left(#2\right)}} 
\newcommand{\nsi}[1]{\left\Vert#1\right\Vert_{S_p(I)}} 
\newcommand{\nw}[2]{\left\Vert#1\right\Vert_{W\left(#2\right)}} 
\newcommand{\nwi}[1]{\left\Vert#1\right\Vert_{W(I)}} 

\newcommand{\scl}[5]{{\frac{1}{{#3}^{#2}}#1\left(\frac{#4}{#3},#5\right)}} 
\newcommand{\scll}[3]{{\frac{1}{{\lambda(t)}^{#2}}#1\left(\frac{#3}{\lambda(t)},t\right)}} 
\newcommand{\scllb}[3]{{\frac{1}{{\lambda(t)}^{#2+1}}\partial_t#1\left(\frac{#3}{\lambda(t)},t\right)}} 
\newcommand{\Scl}[5]{{\frac{1}{{#3}^{#2}}#1\left(\frac{#4}{#3},\frac{#5}{#3}\right)}} 

\newcommand{\partila}{\partial}

\newtheorem{thm}{Theorem}[section]
\newtheorem{lem}[thm]{Lemma}
\newtheorem{pp}[thm]{Proposition}
\newtheorem{cor}[thm]{Corollary}

\theoremstyle{remark}
\newtheorem{rem}[thm]{Remark}

\theoremstyle{definition}
\newtheorem{dfn}[thm]{Definition}

\numberwithin{equation}{section}

\title{Nondispersive radial solutions to energy supercritical 
non-linear wave equations, with applications\footnote{2000 MSC number 35L70}}

\author{Carlos E. Kenig\footnote{The first author was supported in part by NSF 
grant DMS-0456583.}\\ \\
Department of Mathematics\\
University of Chicago\\
Chicago, IL 60637\\
USA\\
cek@math.uchicago.edu \\ \\
Frank Merle\footnote{The second one was supported in part by CNRS. Part of this 
research was carried out during visits of the second author to the University 
of Chicago and IHES. This research was also supported in part by 
ANR ONDENONLIN.}\\
Departement de Mathematiques\\
Universit\'e de Cergy-Pontoise\\
Pontoise\\
95302 Cergy-Pontoise\\
FRANCE\\
frank.merle@math.u-cergy.fr}

\date{}

\begin{document}
\maketitle

\section{Introduction}

In this paper,
we 
consider solutions $u$ to 
\begin{equation}\label{w}
\left\{\begin{array}{l} (\partial_t^2-\triangle)u\pm|u|^{p-1}u=0\quad\quad(x,t)\in\R^3\times\R\\ \\
u\big|_{t=0}=u_0\\ \\
\partial_t u\big|_{t=0}=u_1
\end{array}\right.
\end{equation}
in the range $p\ge5$. We deal with $(u_0,u_1)$ in the scale invariant space $\dot{\mathrm{H}}^{s_p}\times$ $\dot{\mathrm{H}}^{s_p-1}$, $s_p=\frac{3}{2}-\frac{2}{p-1}$. 
(See Definition \ref{d2.10} for the precise definition of solution that we use). We only consider the case when $(u_0,u_1)$ and $u$ are radial. 

We first obtain a pointwise decay estimate for compact radial solutions
to energy critical  and supercritical non-linear wave equations.
We say that $u$ has the ``compactness property'' if it is defined for $t\in(-\infty,+\infty)$ and there exists $\lambda(t)\geq A_0>0$, 
$t\in(-\infty,+\infty)$ so that, for
\begin{displaymath}
K=\left\{\left(\frac{1}{\lambda(t)^{\frac{2}{p-1}}}u\left(\frac{x}{\lambda(t)},t\right),
               \frac{1}{\lambda(t)^{\frac{2}{p-1}+1}}\partial_t u\left(\frac{x}{\lambda(t)},t\right)\right)\;:\;t\in(-\infty,+\infty)\right\}
\end{displaymath}
 $\overline K$ is compact in $\hh$. Our decay estimate then is  (Theorem \ref{t3.2} and Corollary \ref{c3.7} ): There exist $C_0>0$, $r_0>1$ such that, for all 
$t\in\R$, $|x|\geq r_0$, we have 
\begin{equation}\label{e}
\begin{array}{c}
\displaystyle|u(x,t)|\leq \frac{C_0}{|x|}, \\
\displaystyle(\int_{|y|\ge|x|}|\nabla u(y,t)|^2 dy)^{\frac12} 
+|x|^{(p-3)}(\int_{|y|\ge|x|}|\partial_t u(y,t)|^2 dy)^{\frac12}\leq \frac{C_0}{|x|^{\frac12}}.
\end{array}
\end{equation}

Note that, for $p\ge5$,  this estimate ``breaks the scaling''. Note also that it is valid in both the focusing and 
defocusing cases (the $\pm$ signs in (\ref{w})). 
Also, since $W(x)=\left(1+|x|^2/3\right)^{-1/2}$ is a solution in the energy critical
focusing case ($p=5$, $-$ in (\ref{w})), which clearly has $(W(x),0)$ verifying the ``compactness property'', the estimate (\ref{e}) is optimal in the 
range $p\geq5$.

The proof of (\ref{e}) is accomplished by observing that, with $r=|x|$, $w(r,t)=|x|u(x,t)$ solves a non-linear wave equation in one space variable 
\eref{e3.6} and exploits the vector fields $\partial_r+\partial_t$, $\partial_r-\partial_t$ (see Lemma \ref{l3.11}). After this, a ``convexity argument'' using the linear wave equation (Lemma \ref{l3.16})  and an iteration give (\ref{e}).

As a corollary, the solution $u$ which satisfies the compactness property has to be identical to zero in the defocusing situation, for $p\ge5$. 

The importance of solutions satisfying the ``compactness property'' in critical non-linear dispersive and wave equations is by now well established. 
It is  also well understood that estimates like (\ref{e}) are fundamental to obtain,  ``rigidity theorems''. To our knowledge, this is the 
first example of such an estimate in the energy supercritical setting, for non-linear dispersive or wave equations.

Instances where solutions to (\ref{w}) with the ``compactness property'' have been important are, for example, in the study of global well-posedness and 
scattering in both focusing and defocusing cases. In fact, the concentration-compactness/rigidity theorem method that we introduced in \cite{KM1} to study 
global well-posedness and scattering for the focusing, energy critical non-linear Schr\"o\-din\-ger equation in the radial case and which we also applied in 
non-radial settings to the energy critical, focusing nonlinear wave equation  in \cite{KM2} and to $\dot{H}^{1/2}$ bounded solutions of the defocusing 
cubic non-linear Schr\"o\-din\-ger equation in $\R^3$ \cite{KM3}, has, as a major component, a reduction to a ``rigidity theorem'' for solutions having the 
``compactness property''. (The concentration compactness/rigidity theorem method that we introduced in \cite {KM1} has had a wide range of 
applicability by several authors. For instance, it was applied to the ``mass-critical'' NLS in the radial case, in dimension $d\geq2$, in both focusing and 
defocusing cases, in works of Killip, Tao, Vi\c san and Zhang, \cite{TVZ}, \cite{KTV}, \cite{KVZ}, to the focusing cubic NLS for $d=3$, by 
Holmer--Roudenko \cite{HR} and Duyckaerts--Holmer--Roudenko \cite{DHR}, to corotational wave maps into $S^2$ and to $d=4$ Yang--Mills in the radial case, 
in work of C\^ote--Kenig--Merle \cite{CKM}, to the energy critical, focusing NLS for $d\geq5$, in the non-radial case, by Killip--Vi\c san \cite{KV} and 
recently in the work of Tao \cite{T1,T2} on global solutions of the wave map system from $\R^2$ into the hyperbolic plane $H^2$). 
Note that after a first version of our paper was posted in the ArXiv in
October $2008$ ($arXiv:0810.4834$) and motivated by this, Killip and Visan
($arXiv:0812.2084$) have posted corresponding results for defocusing NLS in
dimensions 5 and higher, without radial assumptions, obtained using the
proof of the decay estimate for nondipersive solutions to NLS in high
dimensions that they had previously obtained in their work \cite{KV} in the
energy critical setting in dimensions 5 and higher.

In the last three sections of the paper, we apply estimate (\ref{e}) to radial solutions of (\ref{w}) in the defocusing case , for the energy 
supercritical case $p>5$. We apply our concentration-compactness/ rigidity theorem method, using crucially estimate (\ref{e}), 
to show, (in the spirit of our work \cite{KM3}) that if 
\begin{displaymath}
\sup_{0<t<\tpu}\nhh{u(t)}{\partial_t u(t)}<\infty,
\end{displaymath} 
where $\tpu$ is the ``final  time of existence'' (see Definition \ref{d2.10}), then $\tpu=+\infty$ and $u$ scatters at $+\infty$ 
(see Remark \ref{r2.12} for a definition of scattering). Thus, if $\tpu<+\infty$, we must have 
\begin{equation}\label{e1.3}
\limsup_{0<t<\tpu}\nhh{u(t)}{\partial_tu(t)}=+\infty.
\end{equation}
Note that this is similar to the result in \cite{KM3} and also to the $L^{3,\infty}$ result of Escauriaza--Seregin--\v Sverak \cite{ESS} for 
Navier--Stokes. 
Note that this type of result for the defocusing energy critical case ($p=5$) has a long history. In fact, in this case, the analog of (1.3) always holds and is a consequence of the conservation of energy. For $p=5$, Struwe \cite{Sm} in the radial case and Grillakis \cite{Gm} in the general case, showed that, for regular data, $T_+((u_0,u_1))<+\infty$ is impossible, and Shatah-Struwe \cite{SS2m}, \cite{SS3m} extended this to global well-posedness and preservation of higher regularity for data in $\dot{\mathrm{H}}^{1}\times L^2$, while Bahouri-Shatah \cite{BSm} establish scattering for such data. These result are based on the facts that, for small local energy data, one has global existence and that local energy concentration is excluded from the Morawetz identity \cite{Mm}. The key point here is that both the Morawetz identity and the energy have the same scaling, which is also the scaling of the critical well-posedness space $\dot{\mathrm{H}}^{1}\times L^2$. This point is not available in the energy supercritical case and we are thus forced to proceed differently.

Our proof of the application reduces maters to establishing a ``rigidity theorem'' (Section 4).
In Section 5, we establish the ``rigidity theorem'' for solutions with $T_+(u_0,u_1)$ $<\infty$ (where we follow ideas in \cite{KM1}, \cite{KM2}) and for 
solutions with the ``compactness property''. Here, the decay estimate \eref{e} is fundamental to allow us to use  ideas in \cite{KM2}. Finally, 
in Section 6 we present a general argument (in the spirit of \cite{MM1,MM2}), which shows that the general ``rigidity theorem'' is a consequence of 
the special cases proved in Section 5.

We expect that  estimates in the spirit of \eref{e} will continue to have crucial applications to \eref{w}. Further applications and extensions to 
higher dimensions will appear in future publications.

Acknowledgement: We are grateful to Chengbo Wang for pointing out an error in our statement and application of the Hardy-Littlewood-Sobolev embedding in a previous version of this paper.

\section{The Cauchy Problem}

In this section we will sketch the theory of the local Cauchy problem
\begin{equation}\label{cp}
\left\{\begin{array}{l} (\partial_t^2-\triangle)u \ + \ \mu |u|^{p-1}u=0\quad\quad(x,t)\in\R^3\times\R\\ \\
u\big|_{t=0}=u_0\in\dot H^{s_p}\\ \\
\partial_t u\big|_{t=0}=u_1\in\dot H^{s_p-1}
\end{array}\right.
\end{equation}
where $\mu=1$ (defocusing) or $\mu=-1$ (focusing) and
$$s_p=\frac{3}{2}-\frac{2}{p-1},$$ 
which is the critical index for \eref{cp}. We will concentrate in the energy suprecritical case, $5<p$, with 
the energy critical case $p=5$ being covered in various places, in particular in \cite{KM2}, where references are also given. We say that \eref{cp} 
is $\hh$ critical, because if $u$ is a solution of \eref{cp} and $\lambda>0$, by scaling
$u_\lambda(x,t)=\frac{1}{\lambda^{2/p-1}}u\left(\frac{x}{\lambda},\frac{t}{\lambda}\right)$ is also a solution, with initial data
$(u_{0,\lambda}(x),u_{1,\lambda}(x))   =   \frac{1}{\lambda^{2/p-1}} ( u_0\left(\frac{x}{\lambda}\right), \frac{1}{\lambda}u_1\left(\frac{x}{\lambda}\right))$,
and we have
\begin{displaymath}
\nhh{u_{0,\lambda}}{u_{1,\lambda}}=\nhh{u_{0}}{u_{1}}.
\end{displaymath}

We will start out with some preliminary results that are needed for the theory of the local Cauchy problem.

\begin{lem}[Strichartz estimates \cite{GV}]\label{l2.2}
Consider $w(x,t)$ the solution of the linear Cauchy problem
\begin{equation}\label{e2.3}
\left\{\begin{array}{l} (\partial_t^2-\triangle)w=h\quad\quad(x,t)\in\R^3\times\R\\ \\
w\big|_{t=0}=w_0\\ \\
\partial_t w\big|_{t=0}=w_1
\end{array}\right.
\end{equation}
so that
\begin{displaymath}
w(t)=S(t)(w_0,w_1)+\int_0^t\frac{\sin\left((t-s)\sqrt{-\triangle}\right)}{\sqrt{-\triangle}}h(s)\;ds,
\end{displaymath}
where
\begin{displaymath}
S(t)(w_0,w_1)=\cos\left(t\sqrt{-\triangle}\right)w_0+\left(-\triangle\right)^{-1/2}\sin\left(t\sqrt{-\triangle}\right)w_1.
\end{displaymath}
Then
\begin{multline*}
 \sup_{t\in(-\infty,+\infty)}\nhh{w(t)}{\partial_t w(t)}
 +\nll{D^{s_p-1/2}w}{4}{4}\\
 +\nll{D^{s_p-3/2}\partial_tw}{4}{4}
 +\nll{w}{2(p-1)}{2(p-1)}+\nll{w}{\frac{5}{4}(p-1)}{\frac{5}{2}(p-1)}\\
 \leq C[\nhh{w_0}{w_1}+\nll{D^{s_p-1/2}h}{4/3}{4/3}].
\end{multline*}
\end{lem}

\begin{lem}[Chain rule for fractional derivatives \cite{KPV}]\label{l2.3}
If $F\in C^2$, $F(0)=0$, $F'(0)=$ and $$|F''(a+b)|\leq C
[|F''(a)|+|F''(b)|]
$$ and $$|F'(a+b)|\leq C[|F'(a)|+|F'(b)|],$$
we have, for $0<\alpha<1$
\begin{displaymath}
\nl{D^\alpha F(u)}{p}\leq C\nl{F'(u)}{p_1}\nl{D^\alpha u}{p_2},\quad\frac{1}{p}=\frac{1}{p_1}+\frac{1}{p_2},
\end{displaymath}
\begin{multline*}
\nl{D^\alpha(F(u)-F(v))}{p}\leq C\left[\nl{F'(u)}{p_1}+\nl{F'(v)}{p_1}\right]
\nl{D^\alpha(u-v)}{p_2}+\\ + C\left[\nl{F''(u)}{r_1}+\nl{F''(v)}{r_1}\right]
\left[\nl{D^\alpha u}{r_2}+\nl{D^\alpha v}{r_2}\right]\nl{u-v}{r_3},
\end{multline*}
where $\frac{1}{p}=\frac{1}{r_1}+\frac{1}{r_2}+\frac{1}{r_3}$ and $\quad\frac{1}{p}=\frac{1}{p_1}+\frac{1}{p_2}.$
\end{lem}

We now define the $S_p(I)$, $W(I)$ norms, for  a time interval $I$ by
\begin{displaymath}
\nsi{v}=\nlli{v}{2(p-1)}{2(p-1)}\quad\mathrm{and}\quad \nwi{v}=\nlli{v}{4}{4}.
\end{displaymath}
We now note the following two important consequences of Lemma \ref{l2.3} and the definitions:

Let $F(u)= \pm |u|^{p-1}u$, $5<p<\infty$. Recall that $s_p=\frac{3}{2}-\frac{2}{p-1}$, so that $1<s_p<3/2$. We will also set $\alpha_p=s_p-1/2$, so that
$\frac{1}{2}<\alpha_p<1$. Then:
\begin{equation}\label{e2.4}
\nlli{D^{\alpha_p}F(u)}{4/3}{4/3}\leq C\nsi{u}^{(p-1)}\nwi{D^{\alpha_p}u}
\end{equation}
and
\begin{multline}\label{e2.5}
 \nlli{D^{\alpha_p}(F(u)-F(v))}{4/3}{4/3}\\
 \leq C[\nlli{F'(u)}{2}{2}+\nlli{F'(v)}{2}{2}]\nwi{D^{\alpha_p}(u-v)}\\
 +C[\nlli{F''(u)}{2(p-1)/p-2}{2(p-1)/p-2}+\nlli{F''(v)}{2(p-1)/p-2}{2(p-1)/p-2}]\\
 \times[\nwi{D^{\alpha_p}u}+\nwi{D^{\alpha_p}v}]\nsi{u-v}.
\end{multline}
Recalling also that $|F'(u)|\approx|u|^{p-1}$ and that $|F''(u)|\approx|u|^{p-2}$, using Lemma \ref{l2.2}, \eref{e2.4} and \eref{e2.5}, we obtain, 
in a standard manner (see \cite{Pm}, \cite{GSV}, \cite{SS3m}, \cite{KM1}, \cite{KM2}).
\begin{thm}\label{t2.6}
Assume that $(u_0,u_1)\in\hh$, $5<p<\infty$, $0\in \stackrel{\scriptscriptstyle{\circ}}{I}$,\\ $\nhh{u_0}{u_1}\leq A$. 
Then, there exists $\delta=\delta(A,p)>0$ such that if 
\begin{displaymath}
\left\Vert S(t)(u_0,u_1)\right\Vert_{S_p(I)}<\delta,
\end{displaymath}
there exists a unique solution $u$ to \eref{cp} in $\R^3\times I$ (in the sense of the integral equation), with $(u,\partial_tu)\in C(I;\hh)$,
$$\nsi{u}<2\delta, \ \ \nwi{D^{\alpha_p}u}+\nwi{D^{\alpha_p-1}\partial_tu}<+\infty,$$  and in addition, 
$\nlli{u}{\frac{5}{4}(p-1)}{\frac{5}{2}(p-1)}<\infty.$ \\
Also, if $(u_{0,k},u_{1,k})\to(u_0,u_1)$ as $k\to\infty$ in $\hh$, then $$(u_k,\partial_tu_k)\to(u,\partial_tu)\quad\mathrm{in}\;C(I;\hh),$$ 
where $u_k$ is the
solution corresponding to $(u_{0,k},u_{1,k})$.
\end{thm}

\begin{rem}[Higher regularity of solutions]\label{r2.7}
If $$(u_0,u_1)\in(\dot H^{s_p}\cap\dot H^{s_p+\mu},\dot H^{s_p-1}\cap\dot H^{s_p-1+\mu}),$$ $0\leq\mu\leq1$ and $(u_0,u_1)$ verifies the conditions 
of Theorem \ref{t2.6}, then $(u,\partial_tu)\in C(I;\dot H^{s_p}\cap\dot H^{s_p+\mu}\times\dot H^{s_p-1}\cap\dot H^{s_p-1+\mu})$ and
\begin{displaymath}
\nwi{D^{\alpha_p+\mu}u}+\nwi{D^{\alpha_p}}+\nwi{\partial_tD^{\alpha_p-1+\mu}u}+\nwi{\partial_tD^{\alpha_p-1}u}<\infty,
\end{displaymath}
$\nsi{u}\leq2\delta$. Se \cite{GSV}, for example, for a similar result.
\end{rem}

\begin{rem}\label{r2.8}
There exists $\widetilde\delta=\widetilde\delta_p$ so that, if $\nhh{u_0}{u_1}<\widetilde\delta$, the conclusion of Theorem \ref{t2.6} holds
with $I=\R$. 
This is because of Lemma \ref{l2.2}.
\end{rem}

\begin{rem}\label{r2.9}
Given $(u_0,u_1)\in\hh$, there exists $(0\in)I$ such that the conclusion of Theorem \ref{t2.6} holds. This is because of Lemma \ref{l2.2}.
\end{rem}

\begin{dfn}\label{d2.10}
Let $t_0\in I$. We say that $u$ is a solution of \eref{cp} in $I$ if $(u,\partial_tu)\in C(I;\hh)$, $D^{\alpha_p}u\in W(I)$, $u\in S_p(I)$,
$(u,\partial_tu)\big|_{t=t_0}=(u_0,u_1)$ and the integral equation 
\begin{displaymath}
u(t)=S(t-t_0)((u_0,u_1))-\int_{t_0}^t\frac{\sin\left((t-s)\sqrt{-\triangle}\right)}{\sqrt{-\triangle}}F(u(s))\;ds
\end{displaymath}
holds, with $F(u)=|u|^{p-1}u$, $x\in\R^3$, $t\in I$.
\end{dfn}

It is easy to see that solutions of \eref{cp} are unique (see \cite{Czm} and the argument in 2.10 of \cite{KM1}). This allows us to define a maximal interval 
$I((u_0,u_1))$, where the solution is defined.
\begin{displaymath}
I((u_0,u_1))=\left(t_0-\tmu,t_0+\tpu\right)
\end{displaymath}
and if $I'\subset\subset I((u_0,u_1))$, $t_0\in I'$, then $u$ solves \eref{cp} in $\R^3\times I'$, so that 
$(u,\partial_tu)\in C(I';\hh)$, $D^{\alpha_p}u\in W(I')$, $u\in S_p(I')$, $\partial_tD^{\alpha_p-1}u\in W(I')$ (using \eref{e2.4} and Lemma \ref{l2.2}).

\begin{lem}[Standard finite blow-up criterion]\label{l2.11}
If $\tpu<+\infty$, then
\begin{displaymath}
\ns{u}{[t_0,t_0+\tpu)}=+\infty.
\end{displaymath}
\end{lem}

See \cite{Czm}, \cite{KM1}, Lemma 2.11,  for instance, for a similar proof.

\begin{rem}[Scattering]\label{r2.12}(See \cite{Czm} and remark 2.15 in \cite{KM1} for a similar argument). If $u$ is a solution of \eref{cp} in $\R^3\times I$,
$I=[a,+\infty)$ (or $I=(-\infty, a]$), there exists $(u_0^+,u_1^+)\in\hh$ ($(u_0^-,u_1^-)\in\hh$) so that
\begin{displaymath}
\lim_{t\to+\infty}\nhhx{(u(t),\partial_t u)-\left(S(t)(u_0^+,u_1^+),\partial_tS(t)(u_0^+,u_1^+)\right)}=0
\end{displaymath}
(with a similar statement as $t\downarrow-\infty$). This is a consequence of $\ns{u}{[a,+\infty)}<\infty$.
\end{rem}

We next turn to a perturbation theorem that will be needed for our applications. We first recall an inhomogeneous Strichartz estimate:

\begin{lem}\label{l2.13}
Let $\beta=\theta\alpha_p=\theta(s_p-1/2)=\theta(1-2/p-1)$, where $0<\theta<1$. Define $q$ by $\frac{1}{q}=\frac{1-\theta}{2(p-1)}+\frac{\theta}{4}$.
Assume that $\theta$ is so close to 1 that 
$$q<6
 \ \ \mbox{and} \ \ \frac{4}{q}<1+\frac{1}{2(p-1)}. $$
Define $\widetilde q$ by the equation 
$\frac{1}{2}=\frac{1}{q}+\frac{1}{\widetilde q}$. Then,
\begin{equation}\label{e2.14}
\nlli{D^\beta\int_0^t\frac{\sin\left((t-s)\sqrt{-\triangle}\right)}{\sqrt{-\triangle}}h(s)\;ds}{q}{q}\leq C\nlli{D^\beta h}{\widetilde q'}{\widetilde q'}
\end{equation}
and
\begin{equation}\label{e2.15}
\nsi{\int_0^t\frac{\sin\left((t-s)\sqrt{-\triangle}\right)}{\sqrt{-\triangle}}h(s)\;ds}\leq C\nlli{D^\beta h}{\widetilde q'}{\widetilde q'}.
\end{equation}
\end{lem}

This estimate follows by results of Harmse \cite{H}, Oberlin \cite{O}, Foschi \cite{F}, Vilela \cite{V} and Taggart \cite{Ta}. The version 
we are using here is in Corollary 8.7 of \cite{Ta}, using also Remark 8.4 in the same paper.

\begin{thm}[Perturbation Theorem]\label{t2.16}
Let $I\subset\R$ be  a time interval, $t_0\in I$, $(u_0,u_1)\in\hh$ and some constants $M,A,A'>0$. Let $\widetilde u$ be defined on $\R^3\times I$ 
and satisfy
\begin{displaymath}
\sup_{t\in I}\nhh{\widetilde u(t)}{\partial_t\widetilde u(t)}\leq A,
\end{displaymath}
$\nsi{\widetilde u(t)}\leq M$ and $\nw{D^{\alpha_p}\widetilde u}{I'}<\infty$ for each $I'\subset\subset I$. Assume that 
\begin{displaymath}
(\partial_t^2-\triangle)(\widetilde u)=-F(\widetilde u)+e,\quad\quad(x,t)\in\R^3\times I
\end{displaymath}
(in the sense of the appropriate integral equation) and that
\begin{displaymath}
\nhh{u_0-\widetilde u(t_0)}{u_1-\partial_t\widetilde u(t_0)}\leq A'
\end{displaymath}
and that
\begin{displaymath}
\nlli{D^{\alpha_p}e}{4/3}{4/3}+\nsi{S(t-t_0)\left(u_0-\widetilde u(t_0),u_1-\partial_t\widetilde u(t_0)\right)}\leq\epsilon.
\end{displaymath}
Then there exists $\epsilon_0=\epsilon_0(M,A,A')$ such that there exists a solution $u$ of \eref{cp} in $I$, with 
$(u(t_0),\partial_tu(t_0))=(u_0,u_1)$, for $0<\epsilon<\epsilon_0$, with $\nsi{u}\leq C(M,A,A')$ and for all $t\in I$,
\begin{displaymath}
\nhhx{(u(t),\partial_tu(t))-(\widetilde u(t),\partial_t\widetilde u(t))}\leq C(M,A,A')(A'+\epsilon^\alpha),\;\alpha>0.
\end{displaymath}
\end{thm}

A version of this result, in the context of NLS, was first proved in \cite{CKSTTm}. Other versions for NLS appear in \cite{TVm}. A proof of the corresponding result to Theorem 2.11 for NLS, $p=5$ is given in \cite{K}. Using Lemma \ref{l2.13} it readily extends to our case. We will sketch 
the argument now for the reader's convenience. 

\begin{proof}
In the proof it suffices to consider the case $t_0=0$, $I=[0,L]$, $L<+\infty$ and to assume that $u$ exists and then obtain a priori estimates for it. 
After that, an application of Theorem \ref{t2.6} concludes the proof. The first remark is that
\begin{equation}\label{e2.17}
\nwi{D^{\alpha_p}\widetilde u}\leq\widetilde M,\quad\widetilde M=\widetilde M(M,A').
\end{equation}
To see this, split $I=\bigcup_{j=1}^\gamma I_j$, $\gamma=\gamma(M)$ so that on each $I_j$ we have $\ns{\widetilde u}{I_j}\leq\eta$, 
where $\eta$ is to be determined. Let $I_{j_0}=[a_{j_0},b_{j_0}]$, so that the integral equation gives
\begin{multline*}
\widetilde u(t)=S(t)(\widetilde u(a_{j_0}),\partial_t\widetilde u(a_{j_0}))+
\int_{a_{j_0}}^t\frac{\sin\left((t-s)\sqrt{-\triangle}\right)}{\sqrt{-\triangle}}e\;ds-\\
-\int_{a_{j_0}}^t\frac{\sin\left((t-s)\sqrt{-\triangle}\right)}{\sqrt{-\triangle}}F(\widetilde u)\;ds.
\end{multline*}
We then apply Lemma \ref{l2.2} to obtain
\begin{multline*}
\sup_{t\in I_{j_0}}\nhh{\widetilde u(t)}{\partial_t\widetilde u(t)}+\nw{D^{\alpha_p}\widetilde u}{I_{j_0}}\\
\leq CA+C\left\Vert D^{\alpha_p}e\right\Vert_{L_{I_{j_0}}^{4/3}L_x^{4/3}}+
C\left\Vert D^{\alpha_p}F(\widetilde u)\right\Vert_{L_{I_{j_0}}^{4/3}L_x^{4/3}}\\
\leq CA+C\epsilon+C\ns{\widetilde u}{I_{j_0}}^{p-1}\nw{D^{\alpha_p}\widetilde u}{I_{j_0}},
\end{multline*}
where in the last step we have applied our hypothesis on $e$ and \eref{e2.4}. If we then choose $\eta$ so that $C\eta^{p-1}\leq1/2$, \eref{e2.17} follows.

We now choose $\beta$, $q$ and $\widetilde q$ as in Lemma \ref{l2.13}, so that \eref{e2.14} and \eref{e2.15} hold. We also note the following:
\begin{equation}\label{e2.18}
\nll{D^\beta f}{q}{q}\leq C\nsi{f}^{1-\theta}\nwi{D^{\alpha_p}f}^\theta
\end{equation}
and
\begin{equation}\label{e2.19}
\nll{|f|^{p-1}D^\beta f}{\widetilde q'}{\widetilde q'}\leq C\nsi{f}^{p-1}\nll{D^\beta f}{q}{q}.
\end{equation}
In fact, \eref{e2.18} follows from the inequality
\begin{equation}\label{e2.20}
\nl{D^\beta f}{q}\leq C\nl{f}{2(p-1)}^{1-\theta}\nl{D^{\alpha_p}f}{4}^\theta
\end{equation}
by using H\"older's inequality on $I$. \eref{e2.20} in turn follows from complex interpolation. Moreover, \eref{e2.19} follows also from H\"older's 
inequality, using the definitions of $q$, $\widetilde q$.

To carry out the proof now, note that by \eref{e2.17} and \eref{e2.18} we have $\nlli{D^\beta \widetilde u}{q}{q}$ $\leq\widetilde M$. Also, by \eref{e2.18} 
and our hypothesis we have
\begin{equation}\label{e2.21}
\nlli{D^\beta S(t)\left( (u_0-\widetilde u(0), u_1-\partial_t\widetilde u(0))\right)}{q}{q}\leq\epsilon'
\end{equation}
where $\epsilon'\leq\widetilde M\epsilon^\alpha$.

Write $u=\widetilde u+w$, so that $w$ verifies
\begin{equation}\label{e2.22}
\left\{\begin{array}{l}
\partial_t^2w-\triangle w=-(F(\tilde u+w)-F(\tilde u))-e\\ \\
w\big|_{t=0}=u_0-\widetilde u(0)\\ \\
\partial_tw\big|_{t=0}=u_1-\partial_t\widetilde u(0)
\end{array}\right.
\end{equation}
Split now $I=\bigcup_{j=1}^J I_j$, $J=J(M,\eta)$, so that on each $I_j$ we have
\begin{equation}\label{e2.23}
\ns{\widetilde u}{I_j}+\left\Vert D^\beta\widetilde u\right\Vert_{L_{I_j}^qL_x^q}\leq\eta,
\end{equation}
where $\eta>0$ is to be chosen. Set $I_j=[a_j,a_{j+1})$, $a_0=0$, $a_{J+1}=L$. The integral equation on $I_j$ gives:
\begin{multline}
w(t)= S(t-a_j)(w(a_j,\partial_tw(a_j))
-\int_{a_j}^t\frac{\sin\left((t-s)\sqrt{-\triangle}\right)}{\sqrt{-\triangle}}e(s)\;ds-\\
-\int_{a_j}^t\frac{\sin\left((t-s)\sqrt{-\triangle}\right)}{\sqrt{-\triangle}}\left[F(\widetilde u+w)-F(\widetilde u) \right]\;ds.\label{e2.24}
\end{multline}
We apply \eref{e2.18} and Lemma \ref{l2.2} to the second term and Lemma \ref{l2.13} to the third one. Thus,
\begin{multline}
\ns{w}{I_j}+\left\Vert D^{\beta}w\right\Vert_{L_{I_{j}}^{q}L_x^{q}}\leq\ns{S(t-a_j)(w(a_j),\partial_tw(a_j))}{I_j} \\
+\left\Vert D^{\beta}S(t-a_j)(w(a_j),\partial_tw(a_j))\right\Vert_{L_{I_{j}}^{q}L_x^{q}}+C\epsilon_0\\
+C\left\Vert D^{\beta}\left[F(\widetilde u+w)-F(\widetilde u) \right]\right\Vert_{L_{I_{j}}^{\widetilde q'}L_x^{\widetilde q'}}.
\label{e2.25}\end{multline}

For the last term, we use Lemma \eref{e2.3}, $|F'(u)|\approx|u|^{p-1}$, $|F''(u)\approx|u|^{p-2}$, and H\"older's inequality, to obtain
\begin{multline}
\left\Vert D^{\beta}\left[F(\widetilde u+w)-F(\widetilde u) \right]\right\Vert_{L_{I_{j}}^{\widetilde q'}L_x^{\widetilde q'}}
\leq C[\ns{\widetilde u}{I_j}^{p-1}+\ns{w}{I_j}^{p-1}]\left\Vert D^{\beta}w\right\Vert_{L_{I_{j}}^{q}L_x^{q}}\\
+
\left[\ns{\widetilde u}{I_j}^{p-2}+\ns{w}{I_j}^{p-2}\right]\left[\left\Vert D^{\beta}\widetilde u\right\Vert_{L_{I_{j}}^{q}L_x^{q}}+
\left\Vert D^{\beta}w\right\Vert_{L_{I_{j}}^{q}L_x^{q}}\right]\ns{w}{I_j}\\
\leq C(\eta)[\ns{w}{I_j}+\nllij{D^\beta w}{q}{q}]\\
+ C[\ns{w}{I_j}+\nllij{D^\beta w}{q}{q}]^p,
\label{e2.26}\end{multline}
where $C(\eta)\to0$ with $\eta\to0$.

Combining now \eref{e2.25} and \eref{e2.26}, choosing $\eta$ so small that $C(\eta)\leq1/3$ and defining
\begin{multline*}
\gamma_j=\ns{S(t-a_j)((w(a_j),\partial_tw(a_j)))}{\R}\\+\
+\nllij{D^\beta S(t-a_j)((w(a_j),\partial_tw(a_j)))}{q}{q}
+C\epsilon_0,
\end{multline*}
we see that
\begin{multline}\label{e2.27}
\ns{w}{I_j}+\nllij{D^\beta w}{q}{q}\leq\frac{3}{2}\gamma_j+
C[\ns{w}{I_j}+\nllij{D^\beta w}{q}{q}]^p.
\end{multline}

Note that the choice of $\eta$ depends only on $p$. Now, a standard continuity argument shows that there exists $C_0$, which depends only on $C$ 
(which depends only on $p$) so that, if $\gamma_j\leq C_0$, we have
\begin{equation}\label{e2.28}
\ns{w}{I_j}+\nllij{D^\beta w}{q}{q}\leq3\gamma_j
\end{equation}
and
\begin{equation}\label{e2.29}
C\left[\ns{w}{I_j}+
\nllij{D^\beta w}{q}{q}
\right]^p\leq3\gamma_j.
\end{equation}
Thus, if $\gamma_j\leq C_0$, we have
\begin{multline}\label{e2.30}
\ns{w}{I_j}+\nllij{D^\beta w}{q}{q}\leq 3\left[\ns{S(t-a_j)((w(a_j),\partial_tw(a_j)))}{\R}
\right.\\\left.
+\nllij{D^\beta S(t-a_j)((w(a_j),\partial_tw(a_j)))}{q}{q}
\right]
+3C\epsilon_0.
\end{multline}
To be able to continue in the iteration process, put $t=a_{j+1}$ in \eref{e2.24} and apply $S(t-a_{j+1})$ and use trigonometric identities. We then have
\begin{multline}\label{e2.31}
S(t-a_{j+1})((w(a_{j+1}),\partial_tw(a_{j+1})))=S(t-a_{j})((w(a_{j}),\partial_tw(a_{j})))\\
-\int_{a_j}^{a_{j+1}}\frac{\sin\left((t-s)\sqrt{-\triangle}\right)}{\sqrt{-\triangle}}e(s)\;ds\\
-\int_{a_j}^{a_{j+1}}\frac{\sin\left((t-s)\sqrt{-\triangle}\right)}{\sqrt{-\triangle}}\left[F(\widetilde u+w)-F(\widetilde u)\right]\;ds
\end{multline}
Applying the same argument as before, we see that
\begin{multline}\label{e2.32}
\ns{S(t-a_{j+1})((w(a_{j+1}),\partial_tw(a_{j+1})))}{\R}\\
+\nllij{D^\beta S(t-a_j)((w(a_j),\partial_tw(a_j)))}{q}{q}\\
\leq
\left[\ns{S(t-a_j)((w(a_j),\partial_tw(a_j)))}{\R}\right.\\\left.+
\nllij{D^\beta S(t-a_j)((w(a_j),\partial_tw(a_j)))}{q}{q}
+C\epsilon_0\right]\\
+C(\eta)\left[\ns{w}{I_j}+
\nllij{D^\beta w}{q}{q}
\right]\\+C\left[\ns{w}{I_j}+
\nllij{D^\beta w}{q}{q}
\right]^p.
\end{multline}
Again taking $\eta$ small, depending only on $p$, using \eref{e2.28} and \eref{e2.29}, we find that, if $\gamma_j\leq C_0$, we have $\gamma_{j+1}\leq10\gamma_j$.
Recall that, by assumption and \eref{e2.21}, we have $\gamma_0\leq\epsilon_0'+\epsilon_0$. Iterating, we see that $\gamma_j\leq10^j(\epsilon_0'+\epsilon_0)$
if $\gamma_j\leq C_0$. If we have $\epsilon_0$ so small that $10^{J+1}(\epsilon_0'+\epsilon_0)\leq C_0$, the condition $\gamma_j\leq C_0$ always holds, 
so that using this, together with \eref{e2.28}, we obtain the desired estimate for $\nsi{u}$. The second estimate follows from the first one, using a 
similar argument. This concludes the proof of Theorem \ref{t2.16}.  
\end{proof}

\begin{rem}\label{r2.33}
Theorem \ref{t2.16} yields the following continuity fact: let $(u_0,u_1)\in\hh$, $\nhh{u_0}{u_1}\leq A$, and let $u$ be the solution of \eref{cp}, 
with maximal interval of existence $(-\tmu,\tpu)$. Let $(u_{0,n},u_{1,n})\to(u_0,u_1)$ in $\hh$, and let $u_n$ be the corresponding solution, with maximal 
interval of existence $\left(-T_-((u_{0,n},u_{1,n})), T_+((u_{0,n},u_{1,n}))\right))$. 
Then $\tmu\leq \mbox{liminf} {T_-((u_{0,n},u_{1,n}))}$ and
$\tpu\leq \mbox{liminf} {T_+((u_{0,n},u_{1,n}))}$ as $n$ goes to infinity. Moreover, for each $t\in(-\tmu, \tpu)$, $(u_n(t),\partial_tu_n(t))\to (u(t),\partial_tu(t))$ in $\hh$.
(For the proof see Remark 2.17 in \cite{KM1}).
\end{rem}

\begin{rem}\label{r2.34}
Theorem \ref{t2.16} can also be used to show that if $K\subset\hh$, with $\overline K$ compact in $\hh$, we can find $T_{+,\overline K}>0$, $T_{-,\overline K}>0$,
such that, for all $(u_0,u_1)\in K$, we have $\tpu>T_{+,\overline K}$, $\tmu>T_{-,\overline K}$. Moreover, the family
\begin{displaymath}
\left\{(u(t),\partial_t u(t)):t\in[-T_{-,\overline K},T_{+,\overline K}], (u_0,u_1)\in K\right\}
\end{displaymath}
has compact closure in $C([-T_{-,\overline K},T_{+,\overline K}];\hh)$ and hence it is equicontinuous and bounded.
\end{rem}

We conclude this section with some results that are useful in connection with the finite speed of propagation.

Recall (see \cite{SS}) that, as a consequence of the finite speed of propagation, if $(u_0,u_1)$, $(u_0',u_1')$ verify the conditions of Theorem \ref{t2.6}, 
then the corresponding solutions agree on $\R^3\times I\cap\bigcup_{0\leq t\leq a}\left[B(x_0,(a-t))\times\{t\}\right]$ if $(u_0,u_1)=(u_0',u_1')$ on 
$B(x_0,a)$. This is proved, for instance, in Remark 2.12 of \cite{KM2}, for the case $p=5$, but given the proof of Theorem \ref{t2.6}, it also holds for $5<p<\infty$.
Similar conclusions can be drawn for $t<0$.

\begin{lem}\label{l2.35}
Let $\psi_M$ be radial, $\psi_M\in C^\infty(\R^3)$, with $\psi_M(x)=\psi(\frac{x}{M})$, where $0\leq\psi\leq1$, $\psi\equiv0$ for $|x|<1$,
$\psi\equiv1$ for $|x|\geq2$, and $\psi$ and all its derivatives are bounded. Then there exist a constant $C$, which depends only on $p$ and is 
independent of $M$, so that
\begin{equation}\label{e2.36}
\nhh{\psi_Mv_0}{\psi_Mv_1}\leq C\nhh{v_0}{v_1}.
\end{equation}
\end{lem}

\begin{proof}
By scaling, it suffices to prove \eref{e2.36} for $M=1$. Recall that $1<s_p<3/2$. Set $\alpha=s_p-1$, so that $0<\alpha<1/2$. 
Let us first consider $\psi v_0$. Then, $\nh{\psi v_0}\approx\nl{D^\alpha\nabla(\psi v_0)}{2}$. But,
\begin{displaymath}
\nl{D^\alpha\nabla(\psi v_0)}{2}\leq\nl{D^\alpha(\nabla\psi v_0)}{2}+\nl{D^\alpha(\psi\nabla v_0)}{2}=I+II.
\end{displaymath}
Using theorem A.8 in \cite{KPV},
\begin{displaymath}
I\leq\nl{\nabla\psi\cdot D^\alpha v_0}{2}+\nl{D^\alpha\nabla\psi\cdot v_0}{2}+C\nl{D^\alpha v_0}{q}\nl{\nabla\psi}{r},\quad\frac{1}{p}+\frac{1}{r}=\frac{1}{2}.
\end{displaymath}
Choose $q=6$, $r=3$ and note that, since $\nabla D^\alpha v_0\in L^2$, $D^\alpha v_0\in L^6$. Also, $\nabla\psi\in C_0^\infty(\R^3)$ so 
$\nl{\nabla\psi}{3}<\infty$. This allows us to control the first and last terms in the right hand side. For the second term, note that 
$v_0\in L^{\frac{3}{2}(p-1)}$ by Sobolev embedding, since $v_0\in\dot H^{s_p}$. Using H\"older's inequality, we can control the second term, using 
$\frac{1}{q_1}+\frac{2}{3(p-1)}=\frac{1}{2}$. We thus need to show that if $\phi\in C_0^\infty$, $\nl{D^\alpha\phi}{q_1}\leq\infty$. But it is easy to see, using Fourier transform, that $D^\alpha\phi\in L^2\cap L^\infty$, which gives our bound for I.
\begin{displaymath}
II\leq\nl{D^\alpha\left((1-\psi)\nabla v_0\right)}{2}+\nl{D^\alpha\nabla v_0}{2}.
\end{displaymath}
Thus, if $\phi=(1-\psi)$, $\phi\in C_0^\infty$ and we need to bound $\nl{D^\alpha(\phi\nabla v_0)}{2}$. Again, using Theorem A.12 in \cite{KPV}, we bound this by the sum
\begin{displaymath}
\nl{\phi D^\alpha\nabla v_0}{2}+\nl{(D^\alpha\phi)\nabla v_0}{2}+C\nl{D^\alpha\nabla v_0}{2}\nl{\phi}{\infty}.
\end{displaymath}
Clearly, the first and third term are controlled. For the second one, $\nabla v_0\in L^r$, $\frac{1}{r}=\frac{1}{3}+\frac{2}{3(p-1)}$, 
$r\geq2$ and H\"older's inequality finishes the proof. (Here $r\geq2$ since $p\geq5$).

For the term $\nl{D^\alpha(\psi v_1)}{2}$, we again bound it by $\nl{D^\alpha v_1}{2}+\nl{D^\alpha(\phi v_1)}{2}$ which is easily controlled by using Theorems A.8 and A.12 in \cite{KPV}.

\end{proof}

\begin{cor}\label{c2.37}
Assume that $(v_0,v_1)\in K\subset\hh$, where $\overline K$ is compact in $\hh$. Let $\psi_M$ be as in Lemma \ref{l2.35}. Consider the solution $v_M$, given by Theorem \ref{t2.6}, to \eref{cp}, with initial data $(\psi_M v_0,\psi_M v_1)$. Then, given $\epsilon>0$, small, there exists $M(\epsilon)>0$, such that , for all $M\geq M(\epsilon)$, all $(v_0,v_1)$ in $K$, we have that $v_M$ is globally defined (i.e. $I=(-\infty,+\infty)$) and
\begin{equation}\label{e2.38}
\sup_{\tau\in(-\infty,+\infty)}\nhh{v_M(\tau)}{\partial_t v_M(\tau)}\leq\epsilon.
\end{equation}
\end{cor}

\begin{proof}
Using the compactness of $\overline K$, the bounds in Theorem \ref{t2.6} and Theorem \ref{t2.16} and \eref{e2.36}, it suffices to show that, for fixed $(v_0,v_1)\in\hh$, we have that $\lim_{M\to\infty}\nhh{\psi_M v_0}{\psi_M v_1}=0$. But this is immediate from \eref{e2.36}, using the density of
$C_0^\infty\times C_0^\infty$ in $\hh$.
\end{proof}

\section{Decay estimates for compact, radial solutions}

In this section, $p\ge5$. We establish now our main decay estimates for compact, radial solutions in the case they are globally defined, which show that they ``break the scaling''. \\
Thus, consider a solution $u$ to \eref{cp}  with $(u_0,u_1)$ radially symmmetric. (Solution is understood in the sense of Definition \ref{d2.10}). Because of the proof of Theorem \ref{t2.6}, $u$ is also radially symetric. We will assume that 
$$t_0=0, \ \ \tpu=+\infty, \ \ \tmu=+\infty,$$ 
and that $u$ has the following ``compactness property'' or nondispersive property:
\begin{multline}\label{e3.1}
\text{There exists $\lambda(t)\geq A_0>0$, $t\in(-\infty,\infty)$ such that}\\
K=\left\{\vec v(x,t)=\left(\frac{1}{\lambda(t)^{2/p-1}}u\left(\frac{x}{\lambda(t)},t\right),\right.\right.
\quad\quad\quad\quad\quad\quad\quad\quad\quad\quad\quad\quad\quad\quad\\
\quad\quad\quad\quad\quad\quad\quad\quad\left.\left.
\frac{1}{\lambda(t)^{2/(p-1)+1}}\partial_tu\left(\frac{x}{\lambda(t)},t\right)\right):t\in(-\infty,\infty)\right\}\\
\text{has the property that $\overline K$ is compact in $\hh$}.
\end{multline}
Our main estimate is:

\begin{thm}\label{t3.2}
Let $u$ be a solution to \eref{cp}, with $(u_0,u_1)$ radially symmetric. 
Assume that $t_0=0$, $\tpu=+\infty$, $\tmu=+\infty$ and that $u$ has the ``compactness property'' \eref{e3.1}. \\
Then, there exists constant
$C_0>0$  depending only on $A_0$, $p$, $\overline K$, so that, for all $t\in\R$, we have

\begin{equation}\label{e3.3}
\begin{array}{c}
\displaystyle|u(x,t)|\leq\frac{C_0}{|x|},   \\
\displaystyle(\int_{r\ge|x|}|r\partial_ru(r,t)|^m dr)^{\frac1m} \leq \frac{C_0}{|x|^{1-a}}, \\
\displaystyle(\int_{r\ge|x|}|r\partial_t u(r,t)|^m dr)^{\frac1m}\leq \frac{C_0}{|x|^{p(1-a)}},
\end{array}
\end{equation}
for all $|x|\geq 1$, where $m=\frac{p-1}2=\frac1a$.
\end{thm}

Note that, in the defocusing case, Proposition 5.4 below shows that $u$ as in Theorem 3.1 must be $0$. Nevertheless, Theorem 3.1 is a crucial step in the proof of Proposition 5.4. What the situation is in the focusing case is unclear at the moment. Note that for $p=5$ and in the focusing case, $$W(x)=\left(1+|x|^2/3\right)^{-\frac12}$$ is a non-dispersive solution and thus our estimate is sharp.
%

In order to prove Theorem 3.1, we need some preliminary estimates. In the sequel, $r=|x|$, and we will sometimes, by abuse of notation, write 
$u(r,t)=u(x,t)$, when $u(-,t)$ is radially symmetric.



\begin{lem}\label{l3.4}
Let $\frac12>\beta\geq0$, $\beta=\frac12-\frac1m$. \\
Then, there exists $C=C(\beta)$ such that for all $\phi$ radial in $\R^3$,
\begin{displaymath}
\n{r^{1-\frac2m}\phi}_{L^m} \leq C\n{\phi}_{\dot H^{\beta}}.
\end{displaymath} 
Moreover, if $1\leq\beta<\frac32$, we have for $\beta=\frac32-\frac1m$
$$
\n{r^{1-\frac2m}\partial_r\phi}_{L^m} \leq C\n{\phi}_{\dot H^{\beta}}
,$$
$$
\n{r^{\frac1m}\phi}_{L^\infty} \leq C\n{\phi}_{\dot H^{\beta}}
.$$
\end{lem}

\begin{proof}
Note that the Fourier transform of a radial function $\phi$, in $\R^3$, is given by the formula
$$
\hat\phi(r) = \frac Cr \int_0^\infty sin(rs) \phi(s)sds.
$$

Thus, if $\tilde\phi(s) = s\phi(s)$, extended oddly for $s<0$, $\tilde\phi \in \dot{H}^\beta(\R)$ and the first inequality follows from the one dimensional Sobolev embedding Theorem.\\
For the second inequality, note that if $\phi \in \dot{H}^\beta(\R^3)$, then $\partial_r\phi \in \dot{H}^{\beta-1}(\R^3)$. Indeed, $\partial_r\phi = \frac x{|x|}\nabla\phi$, $\nabla\phi \in \dot{H}^{\beta-1}(\R^3)$ and for $0\leq \gamma \leq 1$, $\phi \rightarrow \tilde{m} \phi$ is a bounded operator on $ \dot{H}^{\gamma}(\R^3)$, where $\tilde{m}$ is a homogenous of degree $0$ function smooth away from the origin (To see this last statement, note that it holds obviously for $\gamma=0$ and also for $\gamma = 1$, using the Hardy inequality $\n{\frac\phi{|x|}}_{L^2} \leq C\n{\phi}_{\dot{H}^1}$. The general case follows by interpolation). Thus, our second inequality follows from this fact and the first inequality.\\
To establish the third inequality, use the fundamental theorem of calculus and Holder's inequality, to obtain
$$
|\phi(r)| \leq  \int_r^\infty |\partial_s\phi(s)| ds    \leq  \left(   \int_r^\infty |s\partial_s\phi(s)|^m ds \right)^{\frac1m} r^{-\frac1m}
,$$
which is our third inequality

\end{proof}

We now start towards the proof of Theorem \ref{t3.2}, for $u$ satisfying \eref{cp}.Let us assume for example that we are in the defocusing case, the focusing case being identical.

\begin{lem}\label{l3.5}
Let $u$ be as in Theorem \ref{t3.2} and let $w(r,t)=ru(r,t)$. Then, $w$, as a function of $(r,t)$, (extended oddly for $r<0$) verifies
\begin{equation}\label{e3.6}
\partial_t^2 w-\partial_r^2 w=-r|u|^{p-1}u.
\end{equation}
Moreover, there exist functions $g_i(r)$, $i=1,2,3$, defined for $r>0$, with $g_i$ non-increasing, $\lim_{r\to\infty}g_i(r)=0$, which depend only on $K$, $A_0$,  so that, for $t\in\R$, $r>0$ we have
\begin{equation}\label{e3.7}
\begin{array}{c}
\displaystyle|u(r,t)|\leq \frac{g_1(r)}{r^a}\\
\displaystyle\left( \int_r^{4r} |\partial_t w(s,t)+\partial_r w(s,t)|^m ds \right)^{\frac1m} \leq g_2(r)\\
\displaystyle\left( \int_r^{4r} |\partial_t w(s,t)-\partial_r w(s,t)|^m ds \right)^{\frac1m} \leq g_3(r)
\end{array}
\end{equation}
where $a= \frac 2 {p-1}, m= \frac {p-1} 2$.
\end{lem}

\begin{proof}
For regular solutions to \eref{cp}, \eref{e3.6} follows by differentiation. The general case follows by approximation, by Remark \ref{r2.7}.

To prove \eref{e3.7}, we start with the first estimate. We first note:
\begin{multline}\label{e3.8}
\text{For $u$ as in Lemma \ref{l3.5}, given $\epsilon>0$, there exists $r_0=r_0(\epsilon)>0$ }\\ 
\text{so that, for al $t\in\R$, $r\geq r_0$, we have } r^a|u(r,t)|\leq\epsilon.
\end{multline}
To establish \eref{e3.8}, define $$v_0(r,t)=\frac{1}{\lambda(t)^a}u\left(\frac{r}{\lambda(t)},t\right)\text{ and }
v_1(r,t)=\frac{1}{\lambda(t)^{a+1}}\partial_t u\left(\frac{r}{\lambda(t)},t\right).$$ 
We now apply \eref{e2.38}, with $\tau=0$, and $(v_0,v_1)=(v_0(r,t).v_1(r,t))$. Then
\begin{displaymath}
\nhh{\psi_Mv_0(r,t)}{\psi_Mv_1(r,t)}\leq\frac{\epsilon}{C},
\end{displaymath}
for $M$ large. Applying now Lemma \ref{l3.4},
we see that 
$|r^{a}\psi_M(r)v_0(r,t)|\leq\epsilon,$ or
\begin{displaymath}
\left|r^{a}\frac{1}{\lambda(t)^a}u\left(\frac{r}{\lambda(t)},t\right)\right|\leq\epsilon\text{ for }r\geq2M.
\end{displaymath}
But, 
if $\alpha=r/\lambda(t)$, we see that $\alpha^a |u(\alpha,t)|\leq\epsilon$ for $\alpha\lambda(t)\geq2M$. 
But if $\alpha\geq2M/A_0$, $\lambda(t)\alpha\geq2M$, establishing \eref{e3.8}.

Define now $$g_1(r,t)=\sup_{\alpha\geq r}|\alpha^au(\alpha,t)|,\quad g_1(r)=\sup_{t\in(-\infty,+\infty)}g_1(r,t).$$ Clearly, for $r_1\leq r_2$ we have 
$g_1(r_2)\leq g_1(r_1)$ and, from \eref{e3.8}, we have $$\lim_{r\to\infty}g_1(r)=0.$$

Arguing in the same way, to establish \eref{e3.7} it suffices to prove: Given $\epsilon>0$, there exists $r_0=r_0(\epsilon)>0$ so that for all $t\in\R$, $r\geq r_0$, we have 

\begin{equation}\label{e3.9}
\left( \int_r^{4r} |\partial_r w(s,t)|^m ds \right)^{\frac1m} \leq \epsilon,
\end{equation}
\begin{equation}\label{e3.10}
\left( \int_r^{4r} |\partial_t w(s,t)|^m ds \right)^{\frac1m} \leq \epsilon.
\end{equation}

To establish \eref{e3.9}, in light of \eref{e3.8} it suffices to give the corresponding estimate for $r\partial_ru(r,t)$. 
Apply now Lemma \ref{l3.4}, with 
$\phi=\partial_r f$, $\beta=s_p=\frac{3}{2}-\frac{2}{p-1}=\frac{3}{2}-\frac{1}{m}$, $f=\psi_M v_0(r,t)$. 
Then,
\begin{displaymath}
\n{r^{1- \frac2m}\partial_r\left(\psi_M v_0(r,t)\right)}_{L^m}\leq C\nh{\psi_Mv_0}.
\end{displaymath}
By \eref{e2.38}, the right hand side is smaller than $\epsilon/2$, for $M$ large. Since $\partial_r(\psi_Mv_0)=\psi_M\partial_r(v_0) + \partial_r(\psi_M)v_0$, and 
$\partial_r\psi_M=\frac{1}{M}\psi'\left(\frac{r}{M}\right)$, with supp $\psi' \subset (1,2)$,the third inequality in  Lemma \ref{l3.4} now gives 
\begin{displaymath}
\n{r^{1- \frac2m}\partial_r\left(\psi_M v_0(r,t)\right)}_{L^m}\leq C\nh{\tilde\psi_Mv_0},
\end{displaymath}
where $\tilde\psi = 1$ for $r>1$, $\tilde\psi = 0$ for $0<r<\frac12$. Hence taking $M$ even larger, we have 
\begin{displaymath}
\n{r^{1- \frac2m}\psi_M\partial_r v_0(r,t)}_{L^m}\leq \epsilon.
\end{displaymath}
Hence, 
$\left( \int_{2M}^{\infty} |(\frac r{\lambda(t)})\frac1{\lambda(t)^a}\partial_ru(\frac r{\lambda(t)},t)|^m dr \right)^{\frac1m} \leq \epsilon$, 
so that 
$$\left(  \int_{\frac{2M}{\lambda(t)}}^{\infty} |\alpha\partial_\alpha u(\alpha,t)|^m d\alpha \right)^{\frac1m} \leq \epsilon.$$ But, $\{\alpha> \frac{2M}{A_0}\} \subset \{\alpha> \frac{2M}{\lambda(t)}\}$
and \eref{e3.9} follows.

For \eref{e3.10}, we argue similarly, using the first inequality in Lemma \ref{l3.4} with 
$\beta=s_p-1$, $m=\frac{p-1}2$, $q=2$, $\phi=\psi_Mv_1(r,t)$, $\partial_t w(r,t)=rv_1(r,t)$.

\end{proof}

\begin{lem}\label{l3.11}
Let $u$ , $w$, $g_i$ be as in Lemma \ref{l3.5}. Then, there exists a constant $C_p>0$ so that
\begin{equation}\label{e3.12}
g_2(r)\leq C_p g_1^p(r)
\end{equation}
\begin{equation}\label{e3.13}
g_3(r)\leq C_p g_1^p(r)
\end{equation}
\end{lem}

\begin{proof}
Let $$z_1(r,t)=\partila_rw(r,t)+\partial_tw(r,t),$$ $$z_2(r,t)=\partila_rw(r,t)-\partial_tw(r,t).$$
Using \eref{e3.6}, we see that
$$\partial_\tau z_1(r_0+\tau,t_0-\tau)=(\partial_{rr}w-\partial_{tt}w)(r_0+\tau,t_0-\tau)=(r_0+\tau)|u|^{p-1}u(r_0+\tau,t_0-\tau),$$ and
$$\partila_\tau z_2(r_0+\tau,t_0+\tau)=(r_0+\tau)|u|^{p-1}u(r_0+\tau,t_0+\tau).$$

Then, $$z_1(r+s,t_0) = z_1(r +s + Mr,t_0-Mr) - \int_0^{Mr}(r+s+\tau)|u|^{p-1}u(r+s+\tau,t_0-\tau)\;d\tau.$$
Fix $r_0>0$.  Choose $r>r_0$ and  $t_0\in\R$, so that 
$$
 g_2(r_0) = \left( \int_0^{3r} |z_1(r+s,t_0)|^m ds \right)^{\frac1m} 
.$$
Then,

\begin{multline*}
g_2(r_0)\leq     
\left( \int_0^{3r} |z_1(r+s+Mr,t_0-Mr)|^m ds \right)^{\frac1m}\\
+\left( \int_0^{3r} \left(\int_0^{Mr}(r+s+\tau)|u|^p(r+s+\tau,t_0-\tau)d\tau\right)^m ds \right)^{\frac1m} \\
\leq  g_2((M+1)r_0)
+g_1^p(r_0)\left( \int_0^{3r} \left(\int_0^{Mr}(r+s+\tau)^{1-ap}d\tau\right)^m ds \right)^{\frac1m}\\
\leq g_2((M+1)r_0)
+Cg_1^p(r_0)\left( \int_0^{3r} (r+s)^{m(2-ap)} ds \right)^{\frac1m}\\
\leq g_2((M+1)r_0)+Cg_1^p(r)
\end{multline*}
since $2-ap=2-\frac{2p}{p-1}=-\frac{2}{p-1}=-\frac1m<0$ or equivalently $m(2-ap)=-1$.

Since $C$ is independent of $M$, letting $M \rightarrow
 \infty$, we obtain \eref{e3.12}. The argument for \eref{e3.13} is similar.
\end{proof}

\begin{cor}\label{c3.14}
Let $u$, $w$, $g_1$ be as in Lemma \ref{l3.5}. Then
\begin{equation}\label{e3.15}
\left\{
\begin{array}{c}
\left( \int_r^{4r} |\partial_r w(s,t)|^m ds \right)^{\frac1m}   \leq C_p g_1^p(r)\\ \\
\left( \int_r^{4r} |\partial_t w(s,t)|^m ds \right)^{\frac1m}    \leq C_p g_1^p(r)
\end{array}
\right.
\end{equation}
\end{cor}
This is an immediate consequence of \eref{e3.12}, \eref{e3.13}.


\begin{lem}\label{l3.16}
Let $u$, $w$, $g_1$ be as in Lemma \ref{l3.5}. Then, there exists $\beta>0$, $r_0$ large, so that, for $r>r_0$ we have
\begin{equation}\label{e3.17}
g_1(r)\leq \frac{C_0}{r^\beta}.
\end{equation}
\end{lem}

\begin{proof}
We again use equation \eref{e3.6}. Using the standard representation formula for solutions of the wave equation, 
in one space dimension (see \cite{SS}), we obtain:
\begin{multline}\label{e3.18}
r_0u(r_0,t_0)=\frac{1}{2}\left[\left(r_0+\frac{r_0}{2}\right)u\left(r_0+\frac{r_0}{2},t_0-\frac{r_0}{2}\right)\right.\\+\left.
\left(r_0-\frac{r_0}{2}\right)u\left(r_0-\frac{r_0}{2},t_0-\frac{r_0}{2}\right)\right]
+\frac{1}{2}\int_{r_0-\frac{r_0}{2}}^{r_0+\frac{r_0}{2}}\alpha\partila_tu\left(\alpha,t_0-\frac{r_0}{2}\right)\;d\alpha\\
+\frac{1}{2}\int_{0}^{\frac{r_0}{2}}\int_{r_0-\left(\frac{r_0}{2}-\tau\right)}^{r_0+\left(\frac{r_0}{2}+\tau\right)}
\alpha|u|^{p-1}u\left(\alpha,t_0-\frac{r_0}{2}+\tau\right)\;d\tau d\alpha.
\end{multline}

Note that the first integral term is bounded by
$$
C\left( \int_{\frac{r_0}2}^{\frac{3r_0}2} |\alpha\partial_t u(\alpha,t_0-\frac{r_0}2)|^m d\alpha \right)^{\frac1m} r_0^{\frac{m-1}{m}} \leq C g_1^p(\frac{r_0}2) r_0^{\frac{m-1}{m}}
$$
by \eref{e3.15}, where we have used Holder's inequality. Notice also that $a+\frac{(m-1)}m=1$, so that, using \eref{e3.18}, we obtain 
\begin{multline*}
r_0|u(r_0,t_0)|\leq\frac{1}{2}\left[\frac{3}{2}r_0\left|u\left(\frac{3}{2}r_0,t_0-\frac{r_0}{2}\right)\right|
+\frac{1}{2}r_0\left|u\left(\frac{1}{2}r_0,t_0-\frac{r_0}{2}\right)\right|\right]\\
+C_pr_0\frac{g_1^p(r_0/2)}{r_0^a}+C_p\frac{r_0^3}{r_0^{ap}}g_1^p(\frac{r_0}{2}),
\end{multline*}
and
\begin{multline*}
r_0|u(r_0,t_0)|\leq\frac{1}{2}\left[\frac{3}{2}r_0\left|u\left(\frac{3}{2}r_0,t_0-\frac{r_0}{2}\right)\right|
+\frac{1}{2}r_0\left|u\left(\frac{1}{2}r_0,t_0-\frac{r_0}{2}\right)\right|\right]\\
+C_p\frac{r_0}{r_0^a}g_1^p(\frac{r_0}{2}).
\end{multline*}

Clearing, we see that
\begin{displaymath}
r_0^a|u(r_0,t_0)|\leq\frac{1}{2}\left[\left(\frac{3}{2}\right)^{1-a}+\left(\frac{1}{2}\right)^{1-a}\right]g_1\left(\frac{r_0}{2}\right)+
C_pg_1(r_0/2)^p,
\end{displaymath}
and thus, as before
\begin{equation}\label{e3.19}
g_1(r_0)\leq\frac{1}{2}\left[\left(\frac{3}{2}\right)^{1-a}+\left(\frac{1}{2}\right)^{1-a}\right]g_1\left(\frac{r_0}{2}\right)+
C_pg_1\left(\frac{r_0}{2}\right)^p
\end{equation}
Note now that an elementary calculus argument shows that
\begin{equation}\label{e3.20}
\frac{1}{2}\left[\left(\frac{3}{2}\right)^{1-a}+\left(\frac{1}{2}\right)^{1-a}\right]=(1-2\theta_p)\quad 0<\theta_p<1.
\end{equation}
Also note that, since $g_1(r_0)\xrightarrow[r_0\to\infty]{}0$, for $r_0$ large we have $C_pg_1\left(\frac{r_0}{2}\right)^{p-1}\leq\theta_p$
and hence
\begin{equation}\label{e3.21}
g_1(r_0)\leq(1-\theta_p)g_1\left(\frac{r_0}{2}\right),\quad r_0\text{ large}.
\end{equation}
A simple iteration now shows that \eref{e3.21} gives \eref{e3.17} for $r \geq r_0$. Lemma \ref{l3.4} gives the estimate for $r\geq1$ .
\end{proof}

We are now ready to conclude the proof of Theorem \ref{t3.2}. This proceeds by an iteration. Let $u$, $w$, $g_i$, $a$, $\beta$ be as in Lemma \ref{l3.5},
Lemma \ref{l3.16}. 

By choosing a possibly smaller $\beta$, we can insure that $a+\beta <1$. Recall that
$$
|u(r_0,t)|\leq\frac{C_0}{{r_0}^{a+\beta}}, 
$$
by 
\eref{e3.17}.  For $0<\gamma<1$ to be chosen, we have
$w(r_0,t) = w(r_0^\gamma,t) - \int_{r_0^\gamma}^{r_0}\partial_rw(s,t) ds$, so that 
$$
|w(r_0,t)| \leq C\frac{r_0^\gamma}{r_0^{\gamma(a+\beta)}} + \left( \int_{r_0^\gamma}^{r_0} |\partial_rw|^m \right)^{\frac1m} r_0^{\frac{m-1}{m}}.
$$
From \eref{e3.15}, we have that, $\left( \int_{r}^{4r} |\partial_rw|^m \right)^{\frac1m}\leq \frac C{r^{\beta p}}$. But then, we have
\begin{displaymath}
\left( \int_{r}^{\infty} |\partial_rw|^m \right)^{\frac1m} = \left( \Sigma_{k \geq0}\int_{2^kr}^{2^{k+1}r} |\partial_rw|^m \right)^{\frac1m} \leq \left( \Sigma_{k \geq0}\int_{2^kr}^{2^{k+2}r} |\partial_rw|^m \right)^{\frac1m}
\end{displaymath}
and so 
\begin{displaymath}
\left( \int_{r}^{\infty} |\partial_rw|^m \right)^{\frac1m} \leq C \left( \Sigma_{k \geq0} 2^{-kmp\beta} \right)^{\frac1m} \frac1{r^{\beta p}} \leq \frac C{r^{\beta p}}.
\end{displaymath}
Thus,
$$
|w(r_0,t)| \leq C\frac{r_0^\gamma}{r_0^{\gamma(a+\beta)}} + \frac  {r_0^{\frac{m-1}{m}}}{r_0^{\gamma \beta p}},
$$
so that
$$
|u(r_0,t)| \leq \frac{C}{r_0^{(1-\gamma)}r_0^{\gamma(a+\beta)}} + \frac{C}{r_0^{a}r_0^{\gamma \beta p}}.
$$
Choose now $\gamma$ so that 
$$
\gamma (a+\beta) + (1-\gamma) = a + \gamma \beta p.
$$
Then, $\gamma = \frac{1-a}{1-a +\beta(p-1)}$
and $0<\gamma<1$, $\gamma p > 1$, so that, if 
$$
\beta' = \gamma (a+\beta) + (1-\gamma) - a =  \gamma \beta p
$$
then $\beta' > \beta$, $a + \beta' < 1$ and $|u(r_0,t)| \leq \frac{C}{r_0^{a + \beta'}}$. Thus, let 
$$
\beta_0=\beta, \    \   \  \beta_{n+1} =  \gamma_n (a+\beta_n) + (1-\gamma_n) - a =  \gamma_n \beta_n p
\   \mbox{where} \  \gamma_n = \frac{1-a}{1-a +\beta_n(p-1)}.
$$
Iterating, we see that for each $n$ we have $|u(r_0,t)| \leq \frac{C_n}{r_0^{a + \beta_n}}$. Note that that $(\beta_n)$ is increasing and bounded, thus $\beta_n \rightarrow \bar\beta = 1-a$ since $\frac{1-a}{1-a +\bar\beta(p-1)} = \frac1p$. Thus, we have shown that, for each $\epsilon>0$, we have, 
$$|u(r_0,t)| \leq \frac{C_\epsilon}{r_0^{1-\epsilon}}.$$
Consider now $\int_r^{4r} |\partial_r w(s,t)| ds$ which is bounded by
$$
\left( \int_r^{4r} |\partial_r w(s,t)|^m ds \right)^{\frac1m} r^{\frac{m-1}m}   \leq C g_1^p(r) r^{\frac{m-1}m} \leq C\frac  {r^{1-a}}{r^{p(1-a-\epsilon)}} \leq \frac  {C}{r^{(p-1)(1-a) -\epsilon p}}.
$$
But $(p-1)(1-a) >0$, so if $\epsilon$ is so small that $(p-1)(1-a) - \epsilon p>0$, we see that $\int_{1}^\infty |\partial_rw|< \infty$, so $|w(r_0,t)|\leq C_p$ hence $$|u(r_0,t)|\leq \frac C{r_0}.$$
But then, $\left( \int_r^{4r} |\partial_t w(s,t)|^m ds \right)^{\frac1m} \leq \frac C{r^{p(1-a)}}$, and since $r\partial_tu = \partial_t w$, using a geometric series again, we see that 
$$
\left( \int_r^{\infty} |\partial_t u(s,t)|^m ds \right)^{\frac1m}\leq \frac C{r^{p(1-a)}}.
$$
Finally, $r\partial_ru = \partial_r w - u$, so that 
$$
\left( \int_r^{4r} |\partial_t w(s,t)|^m ds \right)^{\frac1m}\leq \frac C{r^{p(1-a)}} +  \frac C{r^{1-a}}\leq   \frac C{r^{1-a}}
$$
from which we obtain  Theorem \ref{t3.2}.

As a corollary, we have

\begin{cor}\label{c3.7}
We have the following inequalities
\begin{displaymath}
\left( \int_r^{\infty} |s\partial_r u(s,t)|^2 ds \right)^{\frac12}\leq \frac C{r^{\frac12}},
\end{displaymath}
\begin{displaymath}
\left( \int_r^{\infty} |s\partial_t u(s,t)|^2 ds \right)^{\frac12}\leq \frac C{r^{\frac12 +(p-3)}}.
\end{displaymath}

\end{cor}

\begin{proof}
The corresponding inequalities where the integration is restricted to $(2^k,2^{k+1})$ are a direct consequence of \eref{e3.3} via Holder's inequality. The estimates then follow by summing a geometric series.
\end{proof}


\section{Application, concentration-compactness procedure}

In this section we will state our main application of the decay estimates in Theorem \ref{t3.2} and begin the proof following the concentration-compactness procedure developed by the authors in \cite{KM1}, \cite{KM2}, \cite{KM3}. We now assume in the next tree section, that we are in the defocusing case ($\mu=1$), that is, $u$ is a solution of 
\begin{equation}\label{w+}
\left\{\begin{array}{l} (\partial_t^2-\triangle)u + |u|^{p-1}u=0\quad\quad(x,t)\in\R^3\times\R\\ \\
u\big|_{t=0}=u_0\\ \\
\partial_t u\big|_{t=0}=u_1
\end{array}\right.
\end{equation}
\begin{thm}\label{t4.1}
Suppose that $u$ is a solution to \eref{w+} with radial data $(u_0,u_1)\in\hh$, $p>5$ and maximal interval of existence $I=(-\tmu,$ $\tpu)$. Assume that
\begin{displaymath}
\sup_{0<t<\tpu}\nhh{u(t)}{\partial_tu(t)}=A<+\infty.
\end{displaymath}
Then $\tpu=+\infty$ and $u$ scatters at $t=+\infty$, i.e. $\exists (u_0^+,u_1^+)\in\dot H^{s_p}\times$ $\dot H^{s_p-1}$, so that
\begin{displaymath}
\lim_{t\to+\infty}\nhhx{(u(t),\partial_tu(t))-S(t)((u_0^+,u_1^+))}=0.
\end{displaymath}
\end{thm}

We point out first some immediate consequences of Theorem \ref{t4.1}.

\begin{cor}\label{c4.2}
If $(u_0,u_1)$ is radial $\in\hh$ is such that $\tpu<+\infty$, then
\begin{displaymath}
\varlimsup_{t\uparrow\tpu}\nhh{u(t)}{\partial_tu(t)}=+\infty.
\end{displaymath}
\end{cor}

\begin{cor}\label{c4.3}
The set of radial $(u_0,u_1)\in\hh$ such that 
\begin{displaymath}
\sup_{t\in[0,\tpu)}\nhh{u(t)}{\partial_tu(t)}<\infty
\end{displaymath}
is an open subset of $\hh$.
\end{cor}
\begin{proof}
This is because, for such data, in light of Theorem \ref{t4.1} (and its proof), we have $\ns{u}{[0,+\infty)}<\infty$ and this gives an open set from Theorem \ref{t2.16}
\end{proof}

In order to start the proof of Theorem \ref{t4.1}, we need some definitions, in analogy with \cite{KM3}.

\begin{dfn}\label{d4.4}
For $A>0$, $p>5$,
\begin{multline*}
B(A)=\Big\{(u_0,u_1)\in\hh:\text{ if $u$ is the solution of \eref{w+}, with initial}\\\text{ data $(u_0,u_1)$ at $t=0$, then }
\sup_{t\in[0,\tpu)}\nhh{u(t)}{\partial_tu(t)}\leq A\Big\}.
\end{multline*}
We also set $B(\infty)=\bigcup_{A>0}B(A)$.
\end{dfn}

\begin{dfn}\label{d4.5}
We say that $SC(A)$ holds if for each $(u_0,u_1)\in B(A)$,\\ $\tpu=+\infty$ and $\ns{u}{[0,+\infty)}<\infty$. We also say that $SC(A;(u_0,u_1))$ holds if
$(u_0,u_1)\in B(A)$, $\tpu=+\infty$, $\ns{u}{[0,+\infty)}<\infty$.
\end{dfn}

We can define similarly $B_{rad}(A)$, $B_{rad}(\infty)$, $SC_{rad}(A)$, if we restrict to $(u_0,u_1)$ radial.

By Theorem \ref{t2.6}, Remark \ref{r2.8}, Theorem \ref{t2.16} we see that for $\widetilde\delta_0$  small enough, if $\nhh{u_0}{u_1}<\widetilde\delta_0$, then $SC(C\widetilde\delta_0;(u_0,u_1))$ holds. Hence, there exists $A_0>0$ small enough, such that $SC(A_0)$ holds. Theorem \ref{t4.1} is equivalent to the statement that $SC\rr(A)$ holds for each $A>0$. Similarly, Theorem \ref{t4.1}, without the radial restriction, is equivalent to the statement that 
$SC(A)$ holds for each $A>0$. Thus, if Theorem \ref{t4.1} fails, there exists a critical value $A_C>0$, with the property that if $A<A_C$, $SC\rr(A)$ holds, but if $A>A_C$, $SC\rr(A)$ fails. The concentration-compactness procedure introduced by the authors in \cite{KM1} and \cite{KM3} consists in establishing the following propositions:

\begin{pp}\label{p4.6}
There exists $(u_{0,C},u_{1,C})$ radial such that $SC(A_C,\!(u_{0,C},\!u_{1,C}))$ fails.
\end{pp}

\begin{pp}\label{p4.7}
If $(u_{0,C},u_{1,C})$ is as in Proposition \ref{p4.6}, there exists $\lambda(t)\in\R^+$, for $t\in[0,T_+(u_{0,C},u_{1,C}))$, such that
\begin{multline*}
K=\Bigg\{\vec v(x,t)=\left(\scll{u_C}{a}{x},\scll{\partial_tu_C}{a+1}{x}\right),\\ 0\leq t<T_+(u_{0,C},u_{1,C})\Bigg\}
\end{multline*}
has the property that $\overline K$ is compact in $\hh$. Here $u_C$ is the solution of \eref{w+}, with data $(u_{0,C},u_{1,C})$ at $t=0$.
\end{pp}

\begin{rem}\label{r4.8}
The absence of the parameter $x(t)$ in Proposition \ref{p4.7} comes from the fact we are in the radial setting. (See also Remark 4.23 in \cite{KM1}). 
\end{rem}

A key tool in the proof of Propositions \ref{p4.6} and \ref{p4.7} is the ``profile decomposition'' due to Bahouri--G\'erard \cite{BG}. The profile decomposition was simultaneously discovered by Merle--Vega \cite{MV} in the mass critical NLS for $d=2$ context and later developed by Keraani \cite{Ke} for the energy critical NLS. Here the ``profile decomposition'' is:

\begin{thm}\label{t4.9}
Given $\{(v_{0,n},v_{1,n})\}\subseteq\hh$, with 
$$\nhh{v_{0,n}}{v_{1,n}} \leq A,$$
there exists a sequence $\{(V_{0,j},V_{1,j})\}$ in $\hh$, a subsequence of  
$(v_{0,n},v_{1,n})$ (which we still denote $(v_{0,n},v_{1,n})$) and a sequence of triples $(\lambda_{j,n};x_{j,n};$ $t_{j,n})\in\R^+\times\R^3\times\R$, which are ``orthogonal'', i.e.
\begin{displaymath}
\frac{\lambda_{j,n}}{\lambda_{j',n}}+\frac{\lambda_{j',n}}{\lambda_{j,n}}+\frac{|t_{j,n}-t_{j',n}|}{\lambda_{j,n}}+\frac{|x_{j,n}-x_{j',n}|}{\lambda_{j,n}}
\xrightarrow[n\to\infty]{}+\infty,
\end{displaymath}
for $j\neq j'$; such that for each $J\geq1$, we have
\begin{itemize}
\item[i)] $\displaystyle v_{0,n}=\sum_{j=1}^J\Scl{V_j^l}{a}{\lambda_{j,n}}{.-x_{j,n}}{-t_{j,n}}+w_{0,n}^J$\\
          $\displaystyle v_{1,n}=\sum_{j=1}^J\Scl{\partial_tV_j^l}{a+1}{\lambda_{j,n}}{.-x_{j,n}}{-t_{j,n}}+w_{1,n}^J,$\\
where $V_j^l(x,t)=S(t)((V_{0,j},V_{1,j}))$ ($l$ stands for linear solution)
\item[ii)] $\displaystyle\varlimsup_{n\to\infty}\ns{S(t)((w_{0,n}^J,w_{1,n}^J))}{-\infty,+\infty}\xrightarrow[J\to\infty]{}0$
\item[iii)] For each $J\geq1$ we have\\
$\displaystyle\nh{v_{0,n}}^2=\sum_{j=1}^J\nh{V_{0,j}}^2+\nh{w_{0,n}^J}^2+\epsilon_0^J(n)$\\
$\displaystyle\nH{v_{1,n}}^2=\sum_{j=1}^J\nH{V_{1,j}}^2+\nh{w_{1,n}^J}^2+\epsilon_1^J(n)$\\
where $|\epsilon_0^J(n)|+|\epsilon_1^J(n)|\xrightarrow[n\to\infty]{}0$.
\end{itemize}
\end{thm}

\begin{rem}\label{r4.10}
If $(v_{0,n},v_{1,n})$ are radial, we can choose $(V_{0,j},V_{1,j})$ radial, $x_{j,n}\equiv0$.
\end{rem}

Theorem \ref{t4.9} is proved, for $p=5$ in \cite{BG}. See also Remark 4.4 in \cite{KM2} and Remark 4.23 in \cite{KM1}. The proof of Theorem \ref{t4.9} is identical to the one in \cite{BG} and will be omitted.

Once we have at our disposal Theorem \ref{t2.16} and Theorem \ref{t4.9}, the procedure used in section 3 of \cite{KM3} can be followed to give a proof of Proposition \ref{p4.6} and Proposition \ref{p4.7}. We omit the details.

\begin{cor}\label{c4.11}
There exists a function $g:(0,+\infty)\to[0,\infty)$ such that, for every $(u_0,u_1)\in B\rr(A)$, we have $\ns{u}{[0,+\infty)}\leq g(A)$.
\end{cor}

The proof of Corollary \ref{c4.11} follows from Theorem \ref{t4.1} and Theorem \ref{t4.9} as in Corollary 2 in \cite{BG}, Corollary 4.11 in \cite{Ke}.

We denote $(u_{0,C},u_{1,C})$ as in Proposition \ref{p4.6} a ``critical element''. We now recall some further properties of ``critical elements''.

\begin{rem}\label{r4.12}
Because of the continuity of $(u(t),\partial_tu(t))$ in $\hh$, we can construct $\lambda(t)$ with $\lambda(t)$ continuous in 
$[0, T_+(u_{0,C},u_{1,C}))$. See the proof of Remark 5.4 of \cite{KM1}.
\end{rem}

\begin{lem}\label{l4.13}
Let $u_C$ be a critical element, as in Propositions \ref{p4.6}, \ref{p4.7}. Then, there is a (possibly different) solution $w$, with a corresponding 
$\widetilde\lambda$ (which can also be chosen continuous) and an $A_0>0$, so that $\widetilde\lambda\geq A_0$ for $t\in[0,T_+(w_0,w_1))$,
$\sup_{t\in[0,T_+(w_0,w_1))}\nhh{w(t)}{\partila_tw(t)}\!\!<\!\!\infty$ and $\ns{w}{[0,T_+(w_0,w_1))}=\!+\infty$.
\end{lem}

The proof follows from the arguments in \cite{KM1}, page 670. See also Lemma 3.10 in \cite{KM3} for a similar proof.

\begin{lem}\label{l4.14}
Let $u_C$ be a critical element as in Propositions \ref{p4.6}, \ref{p4.7}. Assume that $T_+(u_{0,C},u_{1,C})<+\infty$. Then,
\begin{equation}\label{e4.15}
0<\frac{C(\overline K)}{T_+(u_{0,C},u_{1,C})-t}\leq\lambda(t).
\end{equation}
\end{lem}

The proof of Lemma \ref{l4.14} is identical to the one of Lemma 4.7 of \cite{KM2} and is omitted.

\begin{lem}\label{l4.16}
Let $u_C$ be as in Lemma \ref{l4.14}. After scaling, assume, without loss of generality, that $T_+(u_{0,C},u_{1,C})=1$. Then, for $0<t<1$, we have:
\begin{equation}\label{e4.17}
\supp\left(u_C(x,t),\partial_tu_C(x,t)\right)\subset B(0,1-t).
\end{equation}
\end{lem}

\begin{proof}
Consider $$(v_0,v_1)=\left(\scll{u_C}{\frac{2}{p-1}}{x},\scllb{u_C}{\frac{2}{p-1}}{x}\right)$$ and $v_{M(\epsilon)}$ as in Corollary \ref{c2.37}.
Let also $$v(\tau)=\scl{u_C}{\frac{2}{p-1}}{\lambda(t)}{x}{t+\frac{\tau}{\lambda(t)}},$$ which is a solution, for 
$0 \leq t+\frac{\tau}{\lambda(t)}<1$. Note that, by the finite speed of propagation (see the comment after Remark \ref{r2.34},
where we fix $t$ as the initial time and consider $\tau$ to be the time variable)
$$v(x,-t\lambda(t))=v_{M(\epsilon)}(x,-t\lambda(t)),\quad |x|\geq2M(\epsilon)+t\lambda(t).$$

Use now 
the Sobolev embedding $L^q \subset \dot{\mathrm{H}}^{s_p-1}$ where $\frac1q = \frac12 - \frac13(\frac12 - \frac2{p-1})$. Then, using \eref{e2.38},
\begin{multline*}
\int_{|x|\geq2M(\epsilon)+t\lambda(t)}\left[\frac{1}{\lambda(t)^{\frac{2}{p-1}+1}}\left|\nabla u_{0,C}\left(\frac{x}{\lambda(t)}\right)\right|\right]^q\\
+\int_{|x|\geq2M(\epsilon)+t\lambda(t)}
\left[\frac{1}{\lambda(t)^{\frac{2}{p-1}+1}}\left|u_{1,C}\left(\frac{x}{\lambda(t)}\right)\right|\right]^q
\leq C\epsilon.
\end{multline*}
After scaling, this becomes
\begin{displaymath}
\int_{|x|\geq\frac{2M(\epsilon)}{\lambda(t)}+t}\left[|\nabla u_{0,C}(x)|^q + |u_{1,C}(x)|^q\right]\;\leq C\epsilon.
\end{displaymath}
Since $\lambda(t)\to\infty$ as $t\to1$, by Lemma \ref{l4.14}, and $\epsilon>0$ is arbitrary, $u_{0,C}\equiv0$ for $|x|\geq1$, 
$u_{1,C}\equiv0$ for $|x|\geq1$. Scaling gives us the corresponding result for $u_C(x,t)$, $\partial_tu_C(x,t)$, $0<t<1$.
\end{proof}

\section{Rigidity Theorem, Part 1}

In the next two sections we conclude the proof of Theorem \ref{t4.1}, by establishing the following ``rigidity theorem'' for solutions of \eref{w+}.

\begin{thm}\label{t5.1}
Let $u$ be a solution of \eref{w+}, $u$ radial, $5< p<\infty$. Assume that $u$ has initial data $(u_0,u_1)\in\hh$. Assume also that there exists 
$\lambda(t)>0$, $t\in(0,\tpu)$, continuous, such that $\ns{u}{[0,\tpu)}=+\infty$ and
\begin{equation}\label{e5.2}
K=\left\{\vec v(x,t)=\left(\scll{u}{a}{x},\scllb{u}{a}{x}\right)\right\}
\end{equation}
has compact closure in $\hh$ and
$\lambda(t)\geq A_0>0$, for all $0<t<\tpu$.   Then, no such $u$ exists.
\end{thm}

Note that in light of Proposition \ref{p4.7}, Remark \ref{r4.12}, Lemma \ref{l4.13}, 
Theorem \ref{t5.1} 
implies Theorem \ref{t4.1}. \\
In order to establish Theorem \ref{t5.1}, we need some well-known identities. See \cite{SS}, 2.3, for their proofs and Struwe's paper \cite{Sm} for the original work in which they were introduced.

\begin{lem}\label{l5.3}
Let $u$ be a solution of \eref{w+}, $\phi\in C_0^\infty(\R^3)$, radial, $\phi\equiv 1$ for $|x|\leq1$, $\supp\phi\subset B(0,2)$. Let $\phi_R(x)=\phi(x/R)$.
Then, for $0<t<\tpu$, the following identities hold:
\begin{itemize}
\item[i)] $\displaystyle\partila_t\int\phi_R\left(\frac{|\nabla u|^2}{2}+\frac{(\partila_tu)^2}{2}+\frac{1}{p+1}|u|^{p+1}\right)=
-\int\partial_tu\nabla\phi_R\nabla u$
\item[ii)] $\displaystyle\partial_t\int(u\partial_tu\phi_R)=\int(\partial_tu)^2\phi_R-\int|\nabla u|^2\phi_R -  \int|u|^{p+1}\phi_R-\int u\nabla u\nabla\phi_R.$
\item[iii)] \begin{multline*}
\partial_t\left(\int x\phi_R\nabla u\partial_tu\right)=-\frac{3}{2}\int\phi_R(\partial_tu)^2+\frac{3}{p+1}\int\phi_R|u|^{p+1}+\\+
\frac{1}{2}\int\phi_R|\nabla u|^2-\int x\cdot\nabla\phi_R(\partial_tu)^2+\int x\cdot\nabla\phi_R\frac{|u|^{p+1}}{p+1}-\\-
\int(\nabla\phi_R\cdot\nabla u)(x\cdot\nabla u)+\frac{1}{2}\int(x\cdot\nabla\phi_R)|\nabla u|^2.
\end{multline*}
\end{itemize}
\end{lem}

\begin{proof}
Note that, since $s_p>1$, by Sobolev embedding, $|\nabla u|,\partila_t u\in L^2_{loc}$ for $t\in I$. Note also that, by Lemma \ref{l3.4}, with 
$\beta=s_p$, $r^2|u|^{p-1}$ is bounded for $t\in I$, and hence, by the usual Hardy inequality $|u|\in L_{loc}^{p+1}$, 
for $t\in I$. One then approximates $u$ by regular solutions, using Remark \ref{r2.7} and then establishes i), ii), iii) by integration by parts. 
A passage to the limit yields Lemma \ref{l5.3}
\end{proof}

We next proceed to establish Theorem \ref{t5.1} in two special cases:  \\
- $\tpu<+\infty$,\\
- The function $\lambda(t)$ is defined for 
$$-\infty=-\tmu<t<\tpu=+\infty, \text{ with }\lambda(t)\geq A_0>0.$$ 
We will then see, in section 6, that, by a general argument, Theorem \ref{t5.1} 
follows from these  special cases.

\begin{pp}\label{p5.4}
There is no $u$ as in Theorem \ref{t5.1}, with $\tpu<+\infty$.
\end{pp}

\begin{proof}
The proof  is in the spirit of the one of Case 1 in the proof of 
Proposition 5.3 in \cite{KM1}.  
We can assume (by scaling) that $\tpu=1$. By Lemma \ref{l4.14} and Lemma \ref{l4.16}, we have 
$$
\lambda(t)\geq\frac{C}{1-t},\quad 
\mbox{and} \quad \supp u,\partial_tu\subset B(0,1-t), \ \ 0<t<1.
$$ 
By compactness
$$\sup_{0<t<1}\nhh{u(t)}{\partial_tu(t)}\leq A,$$ by Lemma \ref{l5.3}i),
\begin{displaymath}
E(u(t),\partial_tu(t))= \int \frac{|\nabla u|^2}{2}+\frac{(\partila_tu)^2}{2}+\frac{1}{p+1}|u|^{p+1}
\end{displaymath}
is finite and constant in $t$, for $0\leq t<1$. But then,
\begin{multline*}
 \frac{1}{2}\int|\nabla u_0|^2+\frac{1}{2}\int u_1^2+\frac{1}{p+1}\int|u_0|^{p+1} \\
\leq C\int_{|x|\leq(1-t)}|\nabla u|^2+|\partial_tu(t)|^2+\frac{|u(t)|^2}{|x|^2}\\
\leq C\int_{|x|\leq(1-t)}\left\{|\nabla u|^2+|\partial_tu(t)|^2\right\}
\end{multline*}
by Hardy's inequality. But $\nabla u, \partial_t u\in\dot H^{s_p-1}\subset L^q$, $\frac{1}{q}=\frac{1}{2}-\frac{1}{3}\left(\frac{1}{2}-\frac{2}{p-1}\right)$,
with uniformly bounded norm in $t$. Since $q>2$, an application of H\"older's inequality shows (letting $t\to1$) that 
$E(u_0,u_1)=0$ and $(u_0,u_1)\equiv(0,0)$, contradicting $\tpu=1$.
\end{proof}

\begin{pp}\label{p5.5}
Assume that $u$ is a radial solution of \eref{w+} such that \\
$\tmu=+\infty$, $\tpu=+\infty$ and there exists $\lambda(t)\geq A_0>0$, for $-\infty<t<\infty$ so that
$$
K=\left\{\vec v(x,t)
=\left(\scll{u}{a}{x},\scllb{u}{a}{x}\right):
-\infty<t<+\infty \right\}
$$
has compact closure in $\hh$. Then $u\equiv0$.
\end{pp}

\begin{proof}
In light of our main result, Theorem \ref{t3.2}, if we define 
\begin{displaymath}
z(t)=\int u\partial_tu+\int x\nabla u\partila_tu,
\end{displaymath}
clearly, 
$z(t)$ is well defined and $|z(t)|\leq C$. Note that, in light of Corollary \ref{c3.7} 
and ii), iii) in 
Lemma \ref{l5.3}, we have
$$
z'(t)=-\frac{1}{2}\int(\partial_t u)^2-\frac{1}{2}\int|\nabla u|^2-\left(1-\frac{3}{p+1}\right)\int|u|^{p+1}.
$$

Note that $\frac{3}{p+1}<1$, that $\int(\partial_tu)^2+\int|\nabla u|^2+\int|u|^{p+1}<\infty$ for each $t$, because of Corollary \ref{c3.7} 
and that 
Lemma \ref{l5.3}i) implies that
$$
E(u(t),\partial_tu(t))\  = \ E((u_0,u_1))
$$
If then $E((u_0,u_1))=\frac{1}{2}\int|\nabla u_0|^2+\frac{1}{2}\int u_1^2+\frac{1}{p+1}\int|u_0|^{p+1}\neq0,$ we have 
\begin{displaymath}
z'(t)\leq-CE((u_0,u_1)).
\end{displaymath}
But then, $$z(0)-z(t)=-\int_0^tz'(s)\;ds\geq tCE((u_0,u_1)),$$ a contradiction for $t>0$, large, since $|z(t)-z(0)|\leq 2C$. This establishes
Proposition \ref{p5.5}.
\end{proof}

\section{Rigidity Theorem, Part 2}

In this section we will conclude the proof of Theorem \ref{t5.1}. Some of the arguments here are inspired by \cite{MM1,MM2}. The argument is general and 
proceeds in a number of steps.

\begin{lem}\label{l6.1}
Let $u$ be as in Theorem \ref{t5.1}, with $\tpu=+\infty$, $\lambda(t)$ continuous, $\lambda(t)\geq A_0>0$. Define 
\begin{displaymath}
\widetilde u(x,t)=\left(\scll{u}{a}{x},\scllb{u}{a}{x}\right),
\end{displaymath}
which is contained in $\overline K$. Let $\{t_n\}_{n=1}^\infty$ be any sequence, with $t_n\to+\infty$. After passing to a subsequence, 
$\widetilde u(x,t_n)\xrightarrow[n\to\infty]{}(v_0,v_1)\in\hh$. Let $v$ be the corresponding solution to \eref{w+}, with data $(v_0,v_1)$. Then, 
$T_+((v_0,v_1))=T_-((v_0,v_1))=+\infty$ and there exists $\widetilde\lambda(\tau)>0$ so that
\begin{displaymath}
\widetilde v(x,\tau)=\left(\scl{v}{a}{\widetilde\lambda(\tau)}{x}{\tau},\scl{\partial_\tau v}{a+1}{\widetilde\lambda(\tau)}{x}{\tau}\right)\in\overline K.
\end{displaymath}
\end{lem}

\begin{proof}
Note first that $(0,0)\not\in\overline K$, because of Theorem \ref{t2.16}, and the fact that $\ns{u}{[0,+\infty)}=+\infty$. Note also that, for each $\tau\in\R$, 
$t_n+\tau/\lambda(t_n)\geq0$, for $n$ large, since $t_n\lambda(t_n)\geq A_0t_n\to+\infty$. We will first show that for 
$\tau\in(-T_-((v_0,v_1)), T_+((v_0,v_1)))$, we can find $\widetilde \lambda(\tau)>0$ so that $\widetilde v(\tau)\in\overline K$. Indeed 
by uniqueness in \eref{w+} and Remark \ref{r2.33}, we have that, for $\tau\in(-T_-((v_0,v_1)),$ $T_+((v_0,v_1)))$,

\begin{multline*}
\left(\scl{u}{a}{\lambda(t_n)}{x}{t_n+\frac{\tau}{\lambda(t_n)}},\scl{\partial_tu}{a+1}{\lambda(t_n)}{x}{t_n+\frac{\tau}{\lambda(t_n)}}\right)\\
\to(v(x,\tau),\partial_\tau v(x,\tau))
\end{multline*}

in $\hh$. Also, since $t_n+\tau/\lambda(t_n)\geq0$ for $n$ large

\begin{multline*}
\Bigg(\scl{u}{a}{\lambda(t_n+\tau/\lambda(t_n))}{x}{t_n+\frac{\tau}{\lambda(t_n)}},\\
\scl{\partila_tu}{a+1}{\lambda(t_n+\tau/\lambda(t_n))}{x}{t_n+\frac{\tau}{\lambda(t_n)}}\Bigg)\\
\to(w_0(\tau),w_1(\tau))\in\overline K,
\end{multline*}

after taking a further subsequence. But then, it is easy to see that 

\begin{multline*}
\Bigg(\left(\frac{\lambda(t_n)}{\lambda(t_n+\tau/\lambda(t_n))}\right)^av\left(\frac{x\lambda(t_n)}{\lambda(t_n+\tau/\lambda(t_n))},\tau\right),\\
\left(\frac{\lambda(t_n)}{\lambda(t_n+\tau/\lambda(t_n))}\right)^{a+1}\partila_\tau v\left(\frac{x\lambda(t_n)}{\lambda(t_n+\tau/\lambda(t_n))},\tau\right)\Bigg)\\\to(w_0(\tau),w_1(\tau))\in\overline K.
\end{multline*}

Hence, since $(w_0(\tau),w_1(\tau))\in\overline K$, so that $(w_0(\tau),w_1(\tau))\neq(0,0)$, we have, for some $M(\tau)$ positive

\begin{equation}\label{e6.2}
\frac{1}{M(\tau)}\leq\frac{\lambda(t_n)}{\lambda(t_n+\tau/\lambda(t_n))}\leq M(\tau).
\end{equation}

(See page 671 of \cite{KM1} for a similar argument). Taking a subsequence, we can assume that

\begin{displaymath}
\frac{1}{\widetilde\lambda(\tau)}=\lim_{n\to\infty}\frac{\lambda(t_n)}{\lambda(t_n+\tau/\lambda(t_n))},
\end{displaymath}

and hence

\begin{displaymath}
\left(\scl{v}{a}{\widetilde\lambda(\tau)}{x}{\tau},\scl{\partial_\tau v}{a+1}{\widetilde\lambda(\tau)}{x}{\tau}\right)\in\overline K.
\end{displaymath}

Finally, by Proposition \ref{p5.4}, Remark \ref{r4.12}, Lemma \ref{l4.14} and Lemma \ref{l4.16},\\ $\tpu=+\infty$, $\tmu=+\infty$.
\end{proof}

\begin{rem}\label{r6.3}
The proof of \eref{e6.2} above also shows that if $\tau_n\to\tau_0$, then
\begin{equation}\label{e6.4}
\frac{1}{M(\tau_0)}\leq\frac{\lambda(t_n)}{\lambda(t_n+\tau_n/\lambda(t_n))}\leq M(\tau_0),
\end{equation}
after taking a subsequence. This is because if we let
\begin{displaymath}
(v_{0,n},v_{1,n})=\left(\scl{u}{a}{\lambda(t_n)}{x}{t_n},\scl{\partial_tu}{a+1}{\lambda(t_n)}{x}{t_n}\right),
\end{displaymath}
$(v_{0,n},v_{1,n})\to(v_0,v_1)$ in $\hh$. If $v_n$ is the corresponding solution to \eref{w+}, since $t_n+\tau_n/\lambda(t_n)\geq0$ for $n$ large,
because of the continuity of the solution map given in Theorem \ref{t2.16},
\begin{displaymath}
(v_n(x,\tau_n),\partila_\tau v_n(x,\tau_n))\to (v(x,\tau_0),\partila_\tau v(x,\tau_0))\quad\text{in }\hh
\end{displaymath}
and
\begin{multline*}
(v_n(x,\tau_n),\partila_\tau v_n(x,\tau_n))\\=
\left(\scl{u}{a}{\lambda(t_n)}{x}{t_n+\frac{\tau_n}{\lambda(t_n)}},\scl{\partial_tu}{a+1}{\lambda(t_n)}{x}{t_n+\frac{\tau_n}{\lambda(t_n)}}\right).
\end{multline*}
Also
\begin{multline*}
\Bigg(\scl{u}{a}{\lambda(t_n+\tau_n/\lambda(t_n))}{x}{t_n+\frac{\tau_n}{\lambda(t_n)}},\\
\scl{\partial_tu}{a+1}{\lambda(t_n+\tau_n/\lambda(t_n))}{x}{t_n+\frac{\tau_n}{\lambda(t_n)}}\Bigg)\\
=\Bigg(\left[\frac{\lambda(t_n)}{\lambda(t_n+\tau/\lambda(t_n))}\right]^av_n\left(\frac{x\lambda(t_n)}{\lambda(t_n+\tau/\lambda(t_n))},\tau_n\right),\\
\left[\frac{\lambda(t_n)}{\lambda(t_n+\tau/\lambda(t_n))}\right]^{a+1}\partila_\tau v_n\left(\frac{x\lambda(t_n)}{\lambda(t_n+\tau/\lambda(t_n))},\tau_n\right)\Bigg)\\ \to(w_0(\tau_0),w_1(\tau_0))\in\overline K,
\end{multline*}
so that $(w_0(\tau_0),w_1(\tau_0))\neq(0,0)$. It is then easy to see that
\begin{multline*}
\Bigg(\left[\frac{\lambda(t_n)}{\lambda(t_n+\tau/\lambda(t_n))}\right]^av\left(\frac{x\lambda(t_n)}{\lambda(t_n+\tau/\lambda(t_n))},\tau_0\right),\\
\left[\frac{\lambda(t_n)}{\lambda(t_n+\tau/\lambda(t_n))}\right]^{a+1}\partila_\tau v\left(\frac{x\lambda(t_n)}{\lambda(t_n+\tau/\lambda(t_n))},\tau_0\right)\Bigg)\\ \to(w_0(\tau_0),w_1(\tau_0)),
\end{multline*}
which, since $(w_0(\tau_0),w_1(\tau_0))\neq(0,0)$, gives \eref{e6.4}.
\end{rem}

\begin{lem}\label{l6.5}
Let $u$, $\lambda$ be as in Theorem \ref{t5.1}, $\tpu=+\infty$, $\lambda(t)\geq A_0>0$. Then,
\begin{equation}\label{e6.6}
\lim_{t\uparrow+\infty}\lambda(t)=+\infty.
\end{equation}
\end{lem}

\begin{proof}
If not, there exists $\{t_n\}_{n=1}^\infty$, $t_n\uparrow+\infty$ and $A_0\leq\lambda_0<+\infty$, so that $\lambda(t_n)\to\lambda_0$. Apply now Lemma 
\ref{l6.1}, so that, after passing to a subsequence $\widetilde u(x,t_n)\to(v_0,v_1)\in\hh$, $v$ is defined on $\R$ and
\begin{displaymath}
\widetilde v(x,\tau)=\left(\scl{v}{a}{\widetilde\lambda(\tau)}{x}{\tau},\scl{\partial_\tau v}{a+1}{\widetilde\lambda(\tau)}{x}{\tau}\right)\in\overline K.
\end{displaymath}
Moreover, by the proof of Lemma \ref{l6.1}, for each $\tau$, (after passing to a subsequence),
\begin{displaymath}
\frac{1}{\widetilde\lambda(\tau)}=\lim_{n\to\infty}\frac{\lambda(t_n)}{\lambda(t_n+\tau/\lambda(t_n))}\leq\frac{\lambda_0}{A_0},
\end{displaymath}
or $\widetilde\lambda(\tau)\geq\frac{A_0}{\lambda_0}>0$. This, however, contradicts Proposition \ref{p5.5}, since $\widetilde v\in\overline K$, so that 
$\widetilde v\neq0$.
\end{proof}

\begin{lem}\label{l6.6}
Let $u$, $\lambda$ be as in Theorem \ref{t5.1} ($\lambda$ continuous, $\lambda(t)\geq A_0>0$), $\tpu=+\infty$. Then, there exists $M_0>0$ so that
\begin{equation}\label{e6.7}
\text{for all }t'\geq t,\;\lambda(t')\geq\frac{1}{M_0}\lambda(t),\;t>0.
\end{equation}
\end{lem}

\begin{proof}
If not, we can find $t_n'>t_n>0$, so that $\frac{\lambda(t_n)}{\lambda(t_n')}\to+\infty$. Since $\lambda\geq A_0$, $\lambda$ is continuous, 
$t_n\to+\infty$. Because of \eref{e6.6}, one can assume, possibly taking a subsequence and changing $t_n'$, that 
$\lambda(t_n')=\min_{t\geq t_n}\lambda(t)$. Consider $(v_0,v_1)$ as in Lemma \ref{l6.1}, so that
\begin{displaymath}
\Bigg(\scl{u}{a}{\lambda(t_n')}{x}{t_n'},\scl{\partial_tu}{a+1}{\lambda(t_n')}{x}{t_n'}\Bigg)\to(v_0,v_1),
\end{displaymath}
$v$ is defined for $\tau\in\R$, $\widetilde v(\tau)\in\overline K$ and (for a subsequence)
\begin{displaymath}
\frac{1}{\widetilde\lambda(\tau)}=\lim_{n\to\infty}\frac{\lambda(t_n')}{\lambda(t_n'+\tau/\lambda(t_n'))}.
\end{displaymath}
We now claim that $(t_n'-t_n)\lambda(t_n')\to+\infty$. If not, $-\tau_n=(t_n'-t_n)\lambda(t_n')\to-\tau_0$ (after taking a further subsequence) and hence, from Remark \ref{r6.3}
\begin{displaymath}
\frac{1}{M(\tau_0)}\leq\frac{\lambda(t_n')}{\lambda(t_n'+\tau_n/\lambda(t_n'))}=\frac{\lambda(t_n')}{\lambda(t_n)}\to0,
\end{displaymath}
a contradiction. But then, for $\tau\in\R$, $n$ large, we have that $t_n'+\tau/\lambda(t_n')\geq t_n$, so that
\begin{displaymath}
\frac{\lambda(t_n')}{\lambda(t_n'+\tau/\lambda(t_n'))}\leq\frac{\lambda(t_n')}{\lambda(t_n')}\leq1
\end{displaymath}
and hence $\widetilde\lambda(\tau)\geq1$. But then, Proposition \ref{p5.5} shows that $\widetilde v\equiv0$, but $\widetilde v\in\overline K$, a contradiction.
\end{proof}

\begin{rem}\label{r6.7}
Let $u$, $\lambda$ be as in Theorem \ref{t5.1}. Define $\lambda_1(t)=\min_{t_1\geq t}\lambda(t_1)$. Then, because of Lemma \ref{l6.6},
\begin{displaymath}
\frac{A_0}{M_0}\leq\frac{1}{M_0}\lambda(t)\leq\lambda_1(t)\leq\lambda(t),
\end{displaymath}
the set 
\begin{displaymath}
\widetilde K=\Bigg\{\left(\scl{u}{a}{\lambda_1(t)}{x}{t},\scl{\partial_tu}{a+1}{\lambda_1(t)}{x}{t}\right)\Bigg\}
\end{displaymath}
also has compact closure in $\hh$. Moreover, since by Lemma \ref{l6.5} $\lim_{t\to\infty}$ $\lambda(t)=+\infty$ and $\lambda$ is continous, it is easy to see that $\lambda_1$ is continuous and $\lim_{t\to\infty}\lambda_1(t)=+\infty$. Moreover $\lambda_1$ is non-decreasing. Choose now $t_n\uparrow+\infty$ so that 
$\lambda_1(t_n)=2^n$.
\end{rem}

\begin{lem}\label{l6.8}
Let $\{t_n\}$ be defined as in Remark \ref{r6.7}. Then, there exists $C_0>0$ so that
\begin{equation}\label{e6.9}
(t_{n+1}-t_n)\lambda_1(t_n)\leq C_0.
\end{equation}
\end{lem}

\begin{proof}
If not, for some subsequence $\{n_i\}_{i=1}^\infty$, we have $(t_{n_i+1}-t_{n_i})\lambda_1(t_{n_i})\to+\infty$. But then, by monotonicity of $\lambda_1$, we have that
\begin{equation}\label{e6.10}
(t_{n_i+1}-t_{n_i})\lambda_1\left(\frac{t_{n_i+1}+t_{n_i}}{2}\right)\to+\infty.
\end{equation}
Note also that by our choice of $t_n$ and monotonicity of $\lambda_1$, for all $t\in[t_n,t_{n+1}]$ we have
\begin{equation}\label{e6.11}
\frac{1}{2}\leq\frac{\lambda_1(t)}{\lambda_1\left(\frac{t_{n+1}+t_{n}}{2}\right)}\leq2.
\end{equation}
Consider now
\begin{multline*}
\left(\scl{u}{a}{\lambda_1\left(\frac{t_{n_i+1}+t_{n_i}}{2}\right)}{x}{\frac{t_{n_i+1}+t_{n_i}}{2}},\right.\\\left.
\scl{\partial_tu}{a+1}{\lambda_1\left(\frac{t_{n_i+1}+t_{n_i}}{2}\right)}{x}{\frac{t_{n_i+1}+t_{n_i}}{2}}\right)\\
\to(v_0,v_1)\in\overline{\widetilde{K}}.
\end{multline*}
(Note  that $(0,0)\not\in\overline{\widetilde{K}}$ because $\ns{u}{0,\infty}=+\infty$ and Theorem \ref{t2.16}). Apply now Lemma \ref{l6.1}. Then 
$v(\tau)$ is defined in $\R$, $\widetilde v(\tau)$ is in $\overline{\widetilde{K}}$ and
\begin{displaymath}
\frac{1}{\widetilde\lambda_1(\tau)}=\lim_{i\to\infty}
\frac{\lambda_1\left(\frac{t_{n_i+1}+t_{n_i}}{2}\right)}
{\lambda_1\left(\frac{t_{n_i+1}+t_{n_i}}{2}+\frac{\tau}{\lambda_1\left(\frac{t_{n_i+1}+t_{n_i}}{2}\right)}\right)}
\end{displaymath}
(after taking a subsequence in $i$). But, \eref{e6.10} gives us that, for a fixed $\tau$, for $i$ large we have
\begin{displaymath}
t_{n_i}\leq\frac{t_{n_i+1}+t_{n_i}}{2}+\frac{\tau}{\lambda_1\left(\frac{t_{n_i+1}+t_{n_i}}{2}\right)}\leq t_{n_i+1}.
\end{displaymath}
Thus, in light of \eref{e6.11}, $\frac{1}{2}\leq\frac{1}{\widetilde\lambda_1(\tau)}\leq2$, which by Proposition \ref{p5.5} gives $(v_0,v_1)=(0,0)$, which 
contradicts $0\not\in\overline{\widetilde{K}}$. Thus, \eref{e6.9} follows.
\end{proof}

The proof of Theorem \ref{t5.1} now follows immediately: because of \eref{e6.9} and the definition of $t_n$, $(t_{n+1}-t_n)\leq C_02^{-n}$, and
$t_n\uparrow+\infty$. But this is a contradiction, since $t_n\leq2C_0$, summing the geometric series.



\end{document}